\documentclass[10pt,reqno]{amsart}
\usepackage{natbib}
\usepackage{amsaddr}
\usepackage{etoolbox}
\patchcmd{\section}{\scshape}{\bfseries\scshape}{}{}
\makeatletter
\renewcommand{\@secnumfont}{\bfseries}
\makeatother

\usepackage[utf8]{inputenc} 
\usepackage[T1]{fontenc}    
\usepackage{hyperref}       
\usepackage{url}            
\usepackage{booktabs}       
\usepackage{amsfonts}       
\usepackage{nicefrac}       
\usepackage{microtype}      
\usepackage[a4paper,margin=3cm]{geometry}

\usepackage{algorithm}
\usepackage{algorithmic}

\usepackage{amsmath}
\usepackage{color}
\usepackage{enumitem}
\usepackage{amssymb}
\usepackage{amsthm}
\usepackage{graphicx}
\usepackage{dsfont}

\usepackage{subfigure}

\renewcommand{\leq}{\leqslant}
\renewcommand{\geq}{\geqslant}

\renewcommand{\ge}{\geqslant}

\def\<{\langle}
\def\>{\rangle}
\def\|{\Vert}

\definecolor{NavyBlue}{rgb}{0.1,0.1,0.6}

\def\eps{\varepsilon}

\newcommand{\esp}[1]{\mathbb{E}\left[#1\right]}

\newcommand{\NRM}[1]{{{\left\| #1\right\|}}} 
\newcommand{\set}[1]{{{\left\{ #1\right\}}}} 

\newcommand{\R}{\mathbb{R}}
\newcommand{\D}{\mathbb{D}}
\renewcommand{\P}{\mathbb{P}}
\newcommand{\cB}{\mathcal{B}}
\newcommand{\cO}{\mathcal{O}}
\newcommand{\cN}{\mathcal{N}}
\newcommand{\cX}{\mathcal{X}}
\newcommand{\cW}{\mathcal{W}}
\newcommand{\cT}{\mathcal{T}}
\newcommand{\cM}{\mathcal{M}}
\newcommand{\cV}{\mathcal{V}}
\newcommand{\cG}{\mathcal{G}}
\newcommand{\M}{\mathbb{M}}

\newcommand{\dd}{{\rm d}}

\newcommand{\cP}{\mathcal{P}}
\newcommand{\cH}{\mathcal{H}}

\newcommand{\cI}{\mathcal{I}}
\newcommand{\N}{\mathbb{N}}
\newcommand{\E}{\mathbb{E}}
\newcommand{\cD}{\mathcal{D}}
\newcommand{\one}{\mathds{1}}

\newcommand{\cL}{\mathcal{L}}

\newcommand{\interior}{{ \rm int}\ }

\hypersetup{colorlinks=true, urlcolor= blue, linkcolor=blue, citecolor=blue}

\usepackage{amsfonts}
\usepackage{amsthm}
\usepackage{amsmath}
\usepackage{amssymb}
\usepackage{mathrsfs}
\usepackage{amscd}
\usepackage{caption}
\usepackage{indentfirst}
\usepackage{cleveref}

\newtheorem{definition}{Definition}
\newtheorem{theorem}{Theorem}
\newtheorem{proposition}{Proposition}
\newtheorem{cor}{Corollary}

\newtheorem{lemma}{Lemma}

\makeatletter
\g@addto@macro{\endabstract}{\@setabstract}
\newcommand{\authorfootnotes}{\renewcommand\thefootnote{\@fnsymbol\c@footnote}}%
\makeatother

\begin{document}
\begin{center}
	\LARGE 
	Sample Optimality and \emph{All-for-all} Strategies in Personalized Federated and Collaborative Learning \par \bigskip
	\normalsize
	\authorfootnotes
	Mathieu Even\textsuperscript{1},
	Laurent Massoulié\textsuperscript{1,2} and
	Kevin Scaman\textsuperscript{1} \par \bigskip
	
	\textsuperscript{1}Inria - Département d’informatique de l’ENS \\
	\smallskip \par
	\textsuperscript{2}MSR-Inria Joint Centre \\
		\par \bigskip
\end{center}

\begin{abstract}
	In personalized Federated Learning, each member of a potentially large set of agents aims to train a model minimizing its loss function averaged over its local data distribution. 
	We study this problem under the lens of stochastic optimization.
	Specifically, we introduce information-theoretic lower bounds on the number of samples required from all agents to approximately minimize the generalization error of a fixed agent.
    We then provide strategies matching these lower bounds,
    in the \emph{all-for-one} and \emph{all-for-all} settings where respectively one or all agents desire to minimize their own local function. Our strategies are based on a \emph{gradient filtering} approach: provided prior knowledge on some notions of distances or discrepancies between local data distributions or functions, a given agent filters and aggregates stochastic gradients received from other agents, in order to achieve an optimal bias-variance trade-off.
\end{abstract}\vspace{-10pt}

\section{Introduction}
A central task in Federated Learning \citep{mcmahan2017fl,kairouz_advances_2019} is the training of a common model from local data sets held by individual agents. A typical application is when users (\emph{e.g.}~mobile phones, hospitals) want to make predictions (\emph{e.g.}~next-word prediction, treatment prescriptions), but each has access to very few data samples, hence the need for collaboration.
As highlighted by many recent works (\emph{e.g.}~\citet{hanzely_lower_2020,mansour_three_2020}), while training a global model yields better statistical efficiency on the combined datasets of all agents by increasing the number of samples linearly in the number of agents, this approach can suffer from a dramatically poor generalization error on local datasets.
A solution to this generalization issue is the training of \emph{personalized} models, a midway between a shared model between agents and models trained locally without any coordination.

An ideal approach would take the best of  both worlds: increased statistical efficiency by using more samples, while keeping local generalization errors low. This raises  the fundamental question: what is the optimal bias/variance tradeoff between personalization and coordination, and how can it be achieved?

We formulate the personalized federated learning problem as follows, studying it under the lens of stochastic optimization~\citep{Bottou2018}. 
Consider $N\in\N^*$ agents denoted by integers $1\leq i\leq N$, each desiring to minimize its own local function $f_i:\R^d\to\R$, while sharing their stochastic gradients.
Since only a limited number of samples are locally available, we focus on \emph{stochastic gradient descent}-like algorithms, where agents each sequentially compute  stochastic gradients $g_i^k$ such that $\esp{g_i^k}=\nabla f_i$. 
In order to reduce the sample complexity, \emph{i.e.}~the number of samples or stochastic gradients required to reach small generalization error, agents thus need to use stochastic gradients from other agents, that are \emph{biased} since in general $\esp{g_i^k}\ne \nabla f_j$.
We  first consider the \emph{all-for-one} objective, where a single agent $i$ wants to minimize its local function (local generalization error), using its own information as well as information from agents $j\ne i$. We then address the \emph{all-for-all} objective where  all agents want to minimize their local function (local generalization error) in parallel using shared information. Our algorithms are based on a \emph{gradient filtering} approach: in both \emph{all-for-one} and \emph{all-for-all} objectives, upon reception of stochastic gradients $(g_j^k)_{j}$, agent $i$ \emph{filters} these gradients and aggregates them using some weights $\lambda_j$ into $\sum_j \lambda_j g_j^k$, in order to achieve some bias/variance trade-off.
\smallskip

\subsection*{Contributions and outline of the paper}
In this paper, we consider oracle models where at each step $k=1,2,\dots$, one (or all) agent(s) may draw a sample according to its (their) local distribution. We aim at computing the number of stochastic gradients sampled from all agents, required to reach a small generalization error in both \emph{all-for-one} and \emph{all-for-all} formulations, in terms of: biases (distances between functions or distributions), regularity  and noise assumptions. The oracle models,  main assumptions and problem formulations are given in Section~\ref{sec:hyp}. Our main contributions are then as follows:

\textbf{\emph{(i)}} In Section~\ref{sec:lower} we prove \emph{information theoretic} lower bounds: to reach a target generalization error $\eps>0$ for a fixed agent $i$, no algorithm can achieve a reduction in the  number of oracle calls by a factor larger than the total number of agents $\eps$-close --in a suitable sense-- to agent $i$.

\textbf{\emph{(ii)}} We next study a weighted gradient averaging algorithm for the \emph{all-for-one} problem, matching this lower bound.

\textbf{\emph{(iii)}} We then propose in Section~\ref{sec:afa} a parallel extension of the simple weighted gradient averaging algorithm that yields an efficient algorithm for  the \emph{all-for-all} problem. In this algorithm, agents compute stochastic gradients at \emph{their} local estimate, and broadcast it to other agents who may use these to update their own estimates. For $x^k=(x_1^k,\ldots,x_N^k)$ where $x_i^k$ is the local estimate of agent $i$ at iteration $k$, updates of the \emph{all-for-all} algorithm write as:
\begin{equation*}
    x^{k+1}=x^k-\eta W g^k\,,
\end{equation*}
where $g^k=(g_1^k,...,g_N^k)$ for an unbiased stochastic gradient~$g_i^k$ of function $f_i$, a step size $\eta$, and a carefully chosen matrix~$W$.
Agents $i$ thus use stochastic gradients that are doubly biased, as gradients of a ``wrong function'' $f_j$ instead of $f_i$ computed at a ``wrong location'' $x^k_j$ instead of $x^k_i$.

\subsection*{Related works}

\emph{Federated Learning} is a paradigm in machine learning where training is done collaboratively among several agents, taking into account privacy constraints \citep{mcmahan2017fl,konecny_federated_2016,kairouz_advances_2019,wang_federated_2019}. A central task is the training of a common model for all agents, for which both \emph{centralized} approaches orchestrated by a server and \emph{decentralized} approaches with no central coordinator \citep{nedic_network_2018} have been considered. The algorithms we propose in this paper are well suited for a  decentralized implementation.

As  observed in~\citet{hanzely_lower_2020}, training a common model for all users can lead to poor generalization on certain tasks such as e.g. next-word prediction. To improve both accuracy and fairness, \emph{personalized} models thus need to be learnt for each agent \citep{li_fair_2020,mohri_agnostic_2019,yu_salvaging_2021}. Approaches to personalization include fine-tuning \citep{cheng_fine-tuning_2021,khodak_adaptive_2019}, transfer learning techniques \citep{tripuraneni_theory_2020,wang_federated_2019}.
\citet{hanzely_lower_2020,fallah_personalized_2020} among others formulate personalization in FL as the  training of local models with a regularization term that enforces collaboration between users. 
We refer the interested reader to \citet{kulkarni_survey_2020} for a broader survey of Personalized Federated Learning.

While the goal of personalization is to minimize local generalization errors, the above cited works do not provide theoretical guarantees over the sample complexity to obtain small local errors, but instead control errors on a regularized problem, in terms of communication rounds or full gradients used, and not in terms of samples used. \citet{deng_adaptive_2020,mansour_three_2020} among others provide generalization errors under a statistical learning framework that depend on VC-dimensions and on distances between each local data distribution and the mixture of all datasets.
\citet{donahue_model-sharing_2020,donahue_optimality_2021} study the bias-variance trade-off between collaboration and personalization for mean estimation in a game-theoretic framework. \citet{chayti2021linear,grimberg_optimal_2021,beaussart_waffle_2021} frame personalization as a stochastic optimization problem with biased gradients and are the works closest to ours. They consider the training of a single agent with biased gradients from another group of agents, i.e. the \emph{all-for-one} problem and obtain performance guarantees in terms of distance between individual function $f_i$ and the average $N^{-1}\sum_j f_j$. In contrast, we obtain more general performance bounds based on distance bounds between all pairs of functions $f_i$, $f_j$ (or equivalently, pairs of local distributions). 
In addition, we prove matching lower bounds. Finally, we also consider the \emph{all-for-all} problem, for which we obtain efficient algorithms, and leverage in that setting the use of weaker bias assumptions than in the \emph{all-for-one} problem.

\section{Problem Statement and Assumptions \label{sec:hyp}}

We now detail our objectives and the necessary technical assumptions. We consider general stochastic gradient methods and formulate our problems, assumptions and algorithms accordingly. Our lower bounds apply to a more specific problem, namely generalization error minimization, or GEM, where functions $f_j$ of all agents $j$ are all obtained from the same loss function $\ell$. These lower bounds a fortiori apply to the more general stochastic gradient setup of our algorithms.

\subsection*{Stochastic (sub)-gradients}\vspace{-3pt}
Let $f_i:\R^d\to\R$, $1\leq i\leq N$ be agent $i$'s local function to be minimized (\emph{e.g.}~the average of a loss function over its local data distribution). At every iteration $k=1,2,\ldots$, agent $i$ may access unbiased \emph{i.i.d.}~estimates $g_i^k(x)$ of $\nabla f_i(x)$ (or of some subgradient $s(x)\in\partial f_i(x)$, depending on the regularity assumptions made):
\begin{equation*}
    \esp{g_i^k(x)} \in \partial f_i(x)\,,\quad x\in\R^d\,,\,k\geq0\,,\,1\leq i\leq N\,.
\end{equation*}
We consider two different objectives. In the \emph{all-for-one} setting, using information from all agents, a fixed agent $i$ desires to minimize its own local function $f_i$ (typically, in collaborative GEM, $f_i$ is the generalization error $f_i(x)=\E_{\cD_i}[\ell(x,\xi_i)]$). In the \emph{all-for-all} setting, all agents want to minimize their own local function, using shared information.

\subsection*{Oracle models}
To specify the information shared between agents via access to stochastic gradients,
we define the following oracles. The \emph{synchronous oracle} (\emph{resp.}~\emph{asynchronous oracle}) lets at every iteration all agents (\emph{resp.}~only one agent) sample a stochastic gradient. After $K$ queries from the synchronous oracle (\emph{resp.}~asynchronous oracle), each agent will have sampled $K$ (\emph{resp.}~$K/N$ on average) stochastic gradients for a total of $NK$ (\emph{resp.}~$K$) in the whole set of agents. In the \emph{all-for-one} setting where a fixed agent $i$ desires to minimize its own local function $f_i$, all agents $j$ can take the same $x_j^k=x_{i}^k$ as argument of their stochastic gradient, resulting in a simple stochastic gradient descent with biased gradients. The difficulty in analyzing \emph{all-for-all} algorithms lies in the fact that stochastic gradients sampled by agent $j$ are doubly biased for agent $i$, being potentially computed at some $x_j^k$ distinct from $x_i^k$.
\begin{center}
\fbox{
\begin{minipage}{10cm}
\begin{center}{\textbf{Synchronous and asynchronous oracles}}\end{center}

At iterations $k=1,2,...$:
\begin{description}
    \item[1] A set of agents $S^k$ is chosen: $S^k=\set{1,\ldots,N}$ for the \emph{synchronous oracle}, $S^k=\set{i_k}$ where $i_k$ is sampled uniformly at random amongst agents, for the \emph{asynchronous oracle}. ;
    \item[2] For all $j\in S^k$, agent $j$ chooses some $y_j^k$ as a (possibly random) function of all previous stochastic gradients and iterates, and samples $g_j^k(y_j^k)$;
    \item[3] All agents can then perform an update using these new stochastic gradients and all previous ones. 
\end{description}
\end{minipage}
}
\end{center}
For fixed target precision $\eps>0$, the objective is to find, using $T_\eps$ samples from all agents in total, queried with synchronous or asynchronous oracles, models with local generalization error $\eps$. We prove matching lower and upper bounds on the \emph{sample complexity} $T_\eps$ in the \emph{all-for-one} and \emph{all-for-all} settings.

\subsection*{Regularity, Bias and Noise Assumptions}
Throughout the paper, we assume that each function $f_i$ is minimized over $\R^d$, and we denote by $x_i^\star$ such a minimizer.

\textbf{\emph{Bias assumptions on function discrepancies.}} For some non-negative weights $(b_{ij})_{1\leq i,j\leq N}\in{\R^+}^{N\times N}$, $(\tilde{b}_{ij})_{1\leq i,j\leq N}\in{\R^+}^{N\times N}$ and $m\ge 0$, either one of the two following assumptions will be made. Note that if differentiable, we have $\esp{g_i^k(x)}=\nabla f_i(x)$.
\begin{enumerate}[label=B.\arabic*]
    \item \label{hyp:bias1} For all $1\leq i,j\leq N$,  optimal model $x_j^\star$ for $j$ generalizes (with generalization error $b_{ij}$) for agent $i$:
    \begin{equation}
        f_i(x_j^\star)-f_i(x_i^\star)\leq b_{ij}\,.\label{eq:hyp:bias1}
    \end{equation}
    \item \label{hyp:bias2} For all $1\leq i,j\leq N$ and for all $x\in\R^d$, $k\geq 0$: 
    \begin{equation}
        \NRM{\esp{g_i^k(x)}-\esp{g_j^k(x)}}^2\leq  \tilde{b}_{ij}\,.\label{eq:hyp:bias2}
    \end{equation}    
\end{enumerate}

\textbf{\emph{Two sets of regularity and noise assumptions.}} 
The two different noise and regularity assumptions we shall consider are  as follows.
\begin{enumerate}[label=N.\arabic*]
    \item \label{hyp:noise2} For all $1\leq i\leq N$, $f_i$ is convex and for some $B>0$, all $k\geq1$ and $x\in\R^d$, 
    \begin{equation}
        \esp{\NRM{g_i^k(x)}^2}\leq B^2\,.\label{eq:hyp:noise2}
    \end{equation}
    \item \label{hyp:noise1} For all $1\leq i\leq N$, $f_i$ $\mu$-strongly convex, $L$-smooth, and we write $\kappa=\frac{L}{\mu}$. For $\sigma>0$, all  $k\geq1$ and $x\in\R^d$:
    \begin{equation}
        \esp{\NRM{g_i^k(x)-\nabla f_i(x)}^2}\leq \sigma^2\,.\label{eq:hyp:noise1}
    \end{equation}
\end{enumerate}
Our regularity assumptions are standard\footnote{In fact our approach could be extended to other regularity assumptions, \emph{e.g.} functions $f_i$ with~$\mu=0$ as in Appendix~\ref{app:afa_convex}, non-convex, or satisfying a Polyak-Lojasiewicz inequality; we do not pursue this here for lack of space. We also generalize \eqref{hyp:bias2} to function dissimilarities such as those in \citet{karimireddy_scaffold_2020}, in Appendix~\ref{app:afo}, \eqref{hyp:bias3}.};
our bias (and gradient dissimilarity) assumptions are less classical, and can be derived from Wassertstein-like distribution-based distances in the GEM setting (Section~\ref{sec:gem} and Appendix~\ref{sec:discuss}).
Assumption~\eqref{hyp:bias2} will be used in the \emph{all-for-one} setting, while the less restrictive~\eqref{hyp:bias1} will be used in the \emph{all-for-all} setting. It quantifies by how much agents' objectives differ. To the best of our knowledge it has not been previously used in Federated Learning.
It is related to \eqref{hyp:bias2} through 
$f_i(x_j^\star)-f_i(x_i^\star)\leq \frac{1}{2\mu}\NRM{\nabla f_i(x_j^\star)-\nabla f_i(x_i^\star)}^2=\frac{1}{2\mu}\NRM{\nabla f_i(x_j^\star)}^2$ under $\mu$-PL or strong convexity assumptions. More generally, we always have $f_i(x_j^\star)-f_i(x_i^\star)\leq \NRM{x_i^\star-x_j^\star}\NRM{\nabla f_i(x_j^\star)}$.
Our lower bounds apply to functions that verify both these bias assumptions, and \emph{a fortiori} verify the assumptions made in our upper bounds where only one of these bias assumptions is required.

Finally, our upper bounds will use the following settings, that we refer to as \textbf{Setting~\ref{setting1}} and \textbf{Setting~\ref{setting2}}:
\begin{enumerate}
    \item Asynchronous oracle is used, assumption~\eqref{hyp:noise2} holds.\label{setting1}
    \item Synchronous oracle is used, assumption~\eqref{hyp:noise1} holds.\label{setting2}
\end{enumerate}

A key instance of our problem is \textbf{\emph{collaborative generalization error minimization (GEM)}}: let $\cD_i$ for $1\leq i\leq N$ be a probability distribution on a set $\Xi$ (agent $i$'s local distribution, \emph{not} its empirical distribution), $\ell:\cX\times \Xi\to\R$ a loss function, and define the following local objective function, that agent $i$ aims at minimizing:
\begin{equation}
	f_i(x)=\E_{\xi_i\sim\cD_i}\left[\ell(x,\xi_i)\right]\,,\quad x\in\cX\,.\label{eq:function_loss}
\end{equation}
Function $f_i$ is thus the  generalization error on agent $i$'s local distribution.
Stochastic gradients are in that case of the form $g_i^k(x)=\nabla_x\ell(x,\xi_i^k)$ where $\xi_i^k\sim\cD_i$. Counting the number of stochastic gradients used in the whole set of agents to reach a precision $\eps$ for $f_i$ thus reduces to computing the number of samples required from all agents to obtain local generalization error $\eps$ for agent $i$. 
In Appendix~\ref{sec:discuss} we discuss how our bias assumptions follow from bounds on distances between distributions $\cD_i$, $\cD_j$ together with practical scenarios where such bounds can be obtained. 

\textbf{Notation:} in the rest of the paper, variables $t$ or $T$ denote the number of stochastic gradients $g_i^k$ sampled (or data item sampled from personal distribution) 
from all agents, while variables $k$ or $K$ denote the iterates of the algorithms or equivalently to the number of oracle calls made.

\section{IT Lower Bounds on the Sample Complexity \label{sec:lower}}

In this section, we prove lower bounds on the sample complexity of the \emph{all-for-one} problem.
Our lower bounds apply to collaborative GEM, i.e. functions $(f_i)_{1\leq i\leq N}$ of the form~\eqref{eq:function_loss}, for some shared loss function $\ell$ and user distributions $\cD_1,\ldots,\cD_N$.

An oracle $\phi:\R^{N\times d}\to\cI$ is a random function that answers some $\phi(x)\in\cI$ where $\cI$ is an information set, for every query $x\in\R^{N\times d}$. We adapt the definitions of~\citet{agarwal2012ITlowerbounds} of sample complexity for \emph{SGD} to our personalization problem.
Formally, the \emph{first-order} oracle we defined in Section~\ref{sec:hyp} (for either synchronous or asynchronous oracles) and that we write as $\phi\big((\cD_i)_{i=1,\ldots,N},\ell\big)$ for shared loss function $\ell$ and user distributions $\cD_1,\ldots,\cD_N$, returns for $x\in\R^{N\times d}$: 
\begin{equation*}    \phi\big((\cD_i)_{i},\ell\big)(x)=\Big(i,\,x_i,\,\xi_i,\,\ell(x_i,\xi_i),\,g_i^k(x_i)\Big)_{i\in S^k}\,,
\end{equation*}
where $\xi_i\sim\cD_i$ and $S^k$ is the set of agents returning a stochastic gradient at step $k$, according to the synchronous or asynchronous oracle in use.
Given distributions and a loss function $\big((\cD_i)_i,\ell\big)$, we denote by $\M$ the set of all methods $\cM=(\cM_T)_{T\geq 0}$: for any $T\geq 0$, $\cM_T$ makes $K$ oracle calls from oracle $\phi\big((\cD_i)_i,\ell\big)$ while using $T$ stochastic gradient samples from all agents ($T=NK$ with the synchronous oracle, $T=K$ with the asynchronous oracle), and returns $x_i^T\in\R^d$ for agent $i$.
For a set $\D$ of couples of distributions and loss function $\big((\cD_j)_j,\ell\big)$ defining functions $(f_i)_{1\leq i\leq N}$, 
we are interested in lower-bounding:
\begin{equation*}
    \inf_{\cM\in\M}\sup_{((\cD_j)_j,\ell)\in\D} \cT_i^{\eps}\Big(\cM,\big((\cD_j)_j,\ell\big)\Big)\,,
\end{equation*}
where $\cT_i^{\eps}\Big(\cM,\big((\cD_j)_j,\ell\big)\Big)$ is the number of samples required from all agents (uniformly divided between agents) to reach generalization error $\eps>0$ for agent~$i$ and writes as:
\begin{align*}
    \inf\set{ T\in\N^* \,\text{such that}\, \esp{f_i(x_i^T)-\min_{x\in\cX}f_i(x)}\leq \eps}\,.
\end{align*}
For $(b_{ij})\in{\R^+}^{N\times N}$, $r>0$, $B>0$, let $\D(r,b,B)$ be the set of all $\big((\cD_i)_{1\leq i \leq N},\ell\big)$ for probability distributions $(\cD_i)_{1\leq i \leq N}$ on a probability space $\Xi\subset\R^D$ such that all functions $f_i$ parameterized by the losses and distributions in this set verify $\NRM{x_i^\star}\leq r$, assumption~\eqref{hyp:noise2}, and bias assumptions \eqref{hyp:bias1} and~\eqref{hyp:bias2} for $b$ and $\tilde{b}_{ij}=(b_{ij}/r)^2$.

Similarly, for $\sigma>0$, $\mu,L>0$, let $\D_\mu^L(r,b,B)$ be the set of all $\big((\cD_i)_{1\leq i \leq N},\ell\big)$ for probability distributions $(\cD_i)_{1\leq i \leq N}$ on a probability space $\Xi\subset\R^D$ such that all functions $f_i$ parameterized by this set verify $\NRM{x_i^\star}\leq r$,  assumption~\eqref{hyp:noise1}, and bias assumptions~\eqref{hyp:bias1} and~\eqref{hyp:bias2} for $b$ and $\tilde{b}=\mu b$ here. 

\begin{theorem}[IT lower bound]\label{thm:lower_afo} 
    Let $\eps\in(0,1/16)$, and assume that either the synchronous or asynchronous oracle is used, that $(b_{ij})$ verifies the triangle inequality $b_{ij}\leq b_{ik}+b_{kj}$ for all $1\leq i,j,k\leq N$. For some constant $C>0$ independent of the problem and any $i\in\{1,\ldots,N\}$:
	\begin{equation*}
		\inf_{\cM\in\M}\!\sup_{((\cD_j)_j,\ell)\in\D(r,b,B)}\!\!\!\!\!\! \cT_i^{\eps}\Big(\cM,\big((\cD_j)_j,\ell\big)\Big) \!\geq\! \frac{Cr^2B^2N}{\eps^{2}\cN_i^\eps(\frac{b}{4})}\,,
	\end{equation*}
	where $\cN_i^\eps(b)=\sum_j \mathds{1}_{\{b_{ij}\leq \eps\}}$ is the number of agents $j$ verifying $b_{ij}\leq \eps$.
	Under strong-convexity and smoothness assumptions for $\mu=L=1/r^2$, 
	we have, for some $C'>0$:
	\begin{equation*}
		\inf_{\cM\in\M}\!\sup_{((\cD_j)_j,\ell)\in\D_{\mu=1/r^2}^{L=1/r^2}(r,b,\sigma)}\!\!\!\!\!\! \cT_i^{\eps}\Big(\cM,\big((\cD_j)_j,\ell\big)\Big) \geq \frac{C'r^2\sigma^2N}{\eps\cN_i^\eps(\frac{b}{4})}\,.
	\end{equation*}
\end{theorem}
The proof of these lower bounds (Appendix~\ref{app:lower}) builds on lower bounds based on Fano's inequality \citep{duchi2013distancebased} for stochastic gradient descent~\cite{agarwal2012ITlowerbounds} or for information limited statistical estimation~\citep{zhang_information-theoretic_2013,DuchR19}, adapted to personalization.

\textbf{Lower-bound interpretation.} Theorem~\ref{thm:lower_afo} states that, given the knowledge of $(b_{ij})$ and $B^2$ (or~$\sigma^2$, $\mu$ and $L$), there exist difficult instances of the problem that satisfy assumption~\eqref{hyp:noise2} (or~\eqref{hyp:noise1}) and \emph{both} bias assumptions~\eqref{hyp:bias1} and~\eqref{hyp:bias2} for $b$, such that the number of samples from all agents (generated through the synchronous or asynchronous oracle) needed to obtain a generalization error of $\eps$ for an agent $i$ is lower-bounded by the right hand sides of the equations in Theorem~\ref{thm:lower_afo}.

\textbf{Collaboration speedup.} The factor $N\times CB^2r^2\eps^{-2}$ is reminiscent of stochastic gradient descent, and is present in~\cite{agarwal2012ITlowerbounds}: without cooperation, this is the sample complexity of \emph{SGD} for a fixed agent (for $T$ samples from all agents, only $T/N$ are taken from a given agent $i$).
Cooperation appears in the factor $1/\cN_i^\eps(b/4)$: the sample complexity is inversely proportional to the number of agents $j$ that have functions (or distributions) similar to that of $i$. One cannot hope for better than a linear speedup proportional to agents $\eps/4$-close to $i$, in the sense that the functions parameterized by the loss function and distributions verify $f_i(x_j^\star)-f_i(x_i^\star)\leq b_{ij}/4$ and $\NRM{\nabla f_i(x)-\nabla f_j(x)}^2\leq \tilde{b}_{ij}$ for all $x\in\R^d$, where $\tilde{b}=(b/r)^2$ in the first case, and $\tilde{b}=b/r^2=\mu b$ in the second. 

\textbf{Bias assumptions and collaborative GEM.} In Section~\ref{sec:gem}, we relate $\NRM{\nabla f_i -\nabla f_j}$ to distribution-based distances $d_\cH(\cD_i,\cD_j)$ in the GEM setting, where $d_\cH$ is a 1-Wassertstein-like distance. We prove that the distributions leading to the functions used to prove Theorem~\ref{thm:lower_afo} verify $d_\cH(\cD_i,\cD_j)\leq b_{ij}$; thus, the distribution based distances also sharply characterize the sample complexity. The requirement for the lower bound to apply that $(b_{ij})$ verifies the triangle inequality is thus quite natural, when translating bias assumptions into distribution-based distances.
 

\section{Weighted Gradient Averaging is optimal in the \emph{All-for-one} setting} 

We now prove that a weighted gradient averaging (WGA) algorithm is optimal in the \emph{all-for-one} setting. Fix the agent~$i$ who desires to minimize its local function $f_i$ (no longer assumed to derive from a loss function). The iterates of the weighted gradient averaging algorithm write as, for both synchronous and asynchronous oracles:
\begin{equation}\label{eq:wga}
    x^{k+1}=x^k-\eta\sum_{j=1}^N \lambda_j G^k(x)_j\,,
\end{equation}
for some step size $\eta>0$ and stochastic vector $\lambda\in\R^N$, and:
\begin{equation}\label{eq:full_stoch}
    G^k(x)_i=\left\{
    \begin{aligned}
    & N\mathds{1}_{\set{i=i_k}} g_{i_k}^k(x)  \quad \text{(asynchronous)}\\
    & g_{i}^k(x) \,\,\,\,\,\quad\quad\quad \text{(synchronous oracle)}\,.
    \end{aligned}\right.
\end{equation}
In the WGA algorithm, all agents involved keep the same local estimates $x_i^k$.
As defined in Section~\ref{sec:hyp}, $(g_i^k)$ are \emph{i.i.d.}~non-biased stochastic (sub)gradients of functions $f_i$.
The WGA algorithm thus defined is more general than previous ones, through the introduction of parameter~$\lambda$. 
In the GEM setting, WGA is equivalent to training a model on the mixture of distributions $(\cD_j)_j$, with weights $(\lambda_j)_j$.
We have the following convergence guarantees in terms of \emph{sample complexity} for the minimization of a function $f_i$ with WGA iterations.
\begin{theorem}\label{thm:wga_sample} Let $\eps>0$, and set $\lambda_j=\frac{\mathds{1}_{\set{b_{ij}<\eps/2}}}{\cN_i^{\eps}(2b)}$, $\cN_i^\eps$ being defined in Theorem~\ref{thm:lower_afo}. Generate $(x^k)_k$ using~\eqref{eq:wga}.
\begin{enumerate}[label=\ref{thm:wga_sample}.\arabic*]
    \item \label{thm:wga_sample1} \textbf{Under Setting~\ref{setting1}.} Assume that for all $k\geq0$ and $i$, $\NRM{x^k-x_i^\star}\leq D$ for some $D>0$ and bias assumption~\eqref{hyp:bias2} holds for $\tilde{b}=(b/D)^2$. For $K=T_\eps$ and for $\eta=\sqrt{\frac{D^2}{2KNB^2\sum_{j=1}^N\lambda_j^2}}$, we have $\esp{\frac{1}{K}\sum_{0\leq k<K}f_i(x^k)-f_i(x_i^\star)}\leq \eps$ using a total number $T_\eps(i)$ of stochastic gradients from all agents of:
    \begin{equation*}
        T_\eps(i) \leq\frac{4D^2B^2}{\eps^2}\frac{N}{\cN_i^\eps(2b)}\,.
        \vspace{-7pt}
    \end{equation*}
    \item \label{thm:wga_sample2} \textbf{Under Setting~\ref{setting2}.} Denote $\kappa=L/\mu$, assume that one of bias assumptions~\eqref{hyp:bias2} or~\eqref{hyp:bias3} holds for $\tilde{b}=\mu b$ (and any $m\geq0$). For $K=T_\eps/N$ and $\eta=\min(1/(2L), \frac{1}{\mu K}\ln(\frac{F_\lambda^0\mu^2K}{L\sigma^2}\sum_j\lambda_j^2))$ where $F_\lambda^0=\sum_j \lambda_j( f_j(x^0)-f_j(x_j^\star))$, we have $\esp{f_i(x^K)-f_i(x_i^\star)}\leq\eps$ using a total number $T_\eps(i)$ of stochastic gradients from all agents of:
    \begin{equation*}
        T_\eps(i) \leq \Tilde{\mathcal{O}}\left(\frac{\kappa\sigma^2}{\mu\eps}\frac{N}{\cN_i^\eps(2b)}\right)\,.
    \end{equation*}
\end{enumerate}
\end{theorem}
The personalization-dependent factor of both these sample complexities (the second factor $\frac{N}{\cN_i^\eps(2b)}$) matches that of the lower bound in Theorem~\ref{thm:lower_afo}. While in Theorem~\ref{thm:wga_sample1} upper and lower bounds (up to constant factors) exactly match, in the strongly convex and smooth case (Theorem~\ref{thm:wga_sample2}), the lower bound is for $\kappa=1$, and in that particular case, lower and upper bounds match; we conjecture that in the more general case, the optimal sample complexity as a linear dependency in $\sqrt{\kappa}$.
In terms of \emph{gradient filtering}, an optimal strategy is thus to filter and keep stochastic gradients sent by agents that are $\eps$-close, in the sense of~\eqref{hyp:bias2}. Theorem~\ref{thm:afa_sample} presents a tuning of the algorithm for a fixed target precision $\eps$; using a \emph{doubling trick} as we do in Appendix~\ref{app:time_adaptive} yields time-adaptive algorithms, that achieve the same sample complexity (up to constant factors) for any $\eps>0$.

In the GEM setting, the proposed WGA algorithm consists in training a model over the distribution $\cD_i^\eps$, a mixture of the distributions $\set{\cD_j\,:\,d_\cH(\cD_i,\cD_j)\leq \eps/2}$, where $d_\cH$ is a distribution-based distance, leading to an increase in the number of samples available by a factor $\cN_i^\eps$ (thus reducing the variance), while keeping the bias of order $\eps$.
\vspace{-5pt}


\section{\emph{All-for-all} Strategies \label{sec:afa}}

We now focus on the more challenging \emph{all-for-all} setting: all agents desire to minimize their local function. We first present a naive approach, that proves to be sample-optimal, but that cannot be used in practice for large scale problems, as it requires a number computations per oracle call for each agent that increases with $N$.
We thus introduce and analyze the \emph{all-for-all} algorithm, the main algorithmic contribution of our paper.

\subsection*{Naive Approach}

Each agent $i$ keeps $N$ shared local models $x_1^k,\ldots,x_N^k$, where $x_j^k$ estimates $x_j^\star$ at iteration $k$ (the knowledge of $x_j^k$ needs to be shared by all agents). At each iteration $k$, when a sample $\xi^k_j$ is obtained at agent $j$, it is used by that agent to compute unbiased estimates of $\nabla f_j(x^k_i)$ for all $i\in[N]$. These are then broadcast to all 
agents $i$ that verify $b_{ij}\leq \eps$ (target precision). Agents then perform $N$ weighted gradient averaging algorithms simultaneously, leading to the following guarantees.
Under the assumptions of Theorem~\ref{thm:wga_sample1}, we have $\max_{1\leq i\leq N}\esp{f_i(x_i^k)-f_i(x_i^\star)}\leq \eps$, with a number of data items sampled from personal distribution $T_\eps$ of:
\begin{equation*}
    T_\eps \leq\frac{4D^2B^2}{\eps^2}\max_{1\leq i\leq N}\frac{N}{\cN_i^\eps(2b)}\,.
\end{equation*}
Under the assumptions of Theorem~\ref{thm:wga_sample2}, we have $\max_{1\leq i\leq N}\esp{f_i(x_i^k)-f_i(x_i^\star)}\leq \eps$, with a number of data item sampled from personal distribution $T_\eps$ of:
\begin{equation*}
    T_\eps \leq \Tilde{\mathcal{O}}\left(\frac{\kappa\sigma^2}{\mu\eps}\max_{1\leq i\leq N}\frac{N}{\cN_i^\eps(2b)}\right)\,.
\end{equation*}
Using the \emph{all-for-one} lower bounds of Theorem~\ref{thm:lower_afo}, this proves to be optimal. Yet, the memory requirements and computation/communication costs of this approach are forbiddingly high for large $N$ (they scale with $\cN_i^\eps$ for agent $i$), hence the \emph{all-for-all} algorithm below, that relaxes the metric considered by controlling the averaged local errors $N^{-1}\sum_{1\leq i\leq N}\esp{f_i(x_i^k)-f_i(x_i^\star)}$ instead of worst-case ones $\max_{1\leq i\leq N}\esp{f_i(x_i^k)-f_i(x_i^\star)}$. Furthermore, relaxing the quantity controlled leads to weaker necessary bias assumptions.

\subsection*{The \emph{All-for-all} algorithm}

\begin{algorithm}[h]
\caption{\emph{All-for-all} algorithm}
\label{algo:afa}
\begin{algorithmic}[1]
\STATE Step size $\eta>0$, matrix $W\in\R^{N\times N}$ 
\vspace{0.5ex}
\STATE Initialization $x_1^0=\ldots=x_N^0\in\R^d$ ($x_i^0$ at agent $i$)
\vspace{0.5ex}
\FOR{$k=0,1,2,\ldots K-1$}
\vspace{0.5ex}
\STATE Agents $j\in S^k$ (activated agents) compute stochastic gradients $g_j^k(x_j^k)$ and broadcast it to all agents $i$ such that $W_{ij}>0$.
\vspace{0.5ex}
\STATE For $i=1,\ldots,N$, update \vspace{-5pt}
\begin{equation*}
    x_i^{k+1}=x_i^k-\eta\sum_{j\in S^k}W_{ij}g^k_j(x^k_j)
    \vspace{-7pt}
\end{equation*}
\ENDFOR \quad Return $x_i^K$ for agent $i$
\end{algorithmic}
\end{algorithm}
\vspace{-10pt}
We now present the \emph{all-for-all} algorithm (AFA), an adaptation of the weighted gradient averaging algorithm to the \emph{all-for-all} setting, where all agents desire to use all stochastic gradients computed. For $1\leq i\leq N$, initialize $x_i^0=x_0\in\R^d$. At iteration $k$, let $x_i^k\in\R^d$ be agent $i$'s current estimate of $x_i^\star$, and denote $x^k=(x_i^k)_{1\leq i\leq N}\in\R^{N\times d}$. For a step size $\eta>0$ and a symmetric non-negative matrix $W\in\R^{N\times N}$, iterates of the \emph{all-for-all} algorithm are generated with Algorithm~\ref{algo:afa}. 
In Theorem~\ref{thm:afa}, we control the averaged local generalization error amongst all agents:
\begin{equation*}
    F^k=\frac{1}{N}\sum_{i=1}^Nf_i(x_i^k)-f_i(x_i^\star)\,,\quad k\geq 0\,.
\end{equation*}
\begin{theorem}[\emph{All-for-all} algorithm]\label{thm:afa}
Let $K>0$, $\eta>0$, and $W$ a symmetric non-negative random matrix of the form $W=\Lambda\Lambda^\top$ for some stochastic matrix $\Lambda=(\lambda_{ij})_{1\leq i,j\leq N}$.
Let $(x_i^k)_{k\geq0,1\leq i\leq N}$ be generated with Algorithm~\ref{algo:afa}. Assume that bias assumption~\eqref{hyp:bias1} holds for some $(b_{ij})$.
\begin{enumerate}[label=\ref{thm:afa}.\alph*]
    \item \label{thm:afa1} \textbf{Under Setting \ref{setting1}.} Using step size $\eta=\sqrt{\frac{2N\NRM{x^0-x^\Lambda}^2}{KB^2\sum_{i,j}\lambda_{ij}^2}}$:
    \begin{align*}
        \esp{\frac{1}{K}\sum_{k=0}^{K-1}F^k}\!
        &\!\leq \sqrt{\frac{2B^2\sum_{i=1}^N\NRM{x_i^0-x_i^\Lambda}^2}{NK}\!\!\!\!\!\sum_{1\leq i,j\leq N}\lambda_{ij}^2} \\
        &+\frac{1}{N}\sum_{1\leq i,j \leq N} \lambda_{ij}b_{ij}\,,
    \end{align*}
    \item \label{thm:afa2} \textbf{Under Setting \ref{setting2}.} If $\eta=\frac{1}{\mu K}\ln(\frac{NF^0\mu^2K}{L\sigma^2\sum_{ij}\lambda_{ij}^2})\wedge \frac{1}{2L}$:
    \begin{align*}
        \esp{F^K}&\leq F^0e^{- \frac{K}{2\kappa}} + \Tilde{\cO}\left(\frac{L\sigma^2}{K\mu^2N}\sum_{1\leq i,j\leq N}\lambda_{ij}^2 \right)\\
        &+\frac{1}{N}\sum_{1\leq i,j\leq N} \lambda_{ij}b_{ij}\,.
    \end{align*}
\end{enumerate}
In \ref{thm:afa1}, $x^\Lambda$ is a minimizer of the function~\eqref{eq:fLambda}. In terms of sample complexity, we have the following.
\end{theorem}
\begin{theorem}[\emph{All-for-all} sample complexity]\label{thm:afa_sample} Let $\eps>0$, and set $W=\Lambda\Lambda^\top$ for $\lambda_{ij}=\frac{\mathds{1}_\set{b_{ij}<\eps/2}}{\cN_i^\eps(2b)}$.
\begin{enumerate}
    \item Under the same assumptions as in Theorem~\ref{thm:afa1}, we have $\esp{\frac{1}{K}\sum_{k=0}^{K-1}F^k}\leq\eps$ for a total number $T_\eps$ of data item sampled from personal distribution from all agents of:
    \begin{equation*}
        T_\eps\leq \frac{4D^2B^2}{\eps^2}\sum_{i=1}^N\frac{1}{\cN_i^\eps(2b)}\,,
    \end{equation*}
    where $D^2$ bounds all $\NRM{x_i^0-x_i^\Lambda}^2$, $1\leq i\leq N$.
    \item Under the same assumptions as in Theorem~\ref{thm:afa2}, we have $\esp{F^{K}}\leq\eps$ for a total number $T_\eps$ of data item sampled from personal distribution from all agents of:
    \begin{equation*}
        T_\eps\leq 2\max\left( \frac{\kappa\sigma^2}{\eps\mu}\sum_{i=1}^N \frac{1}{\cN_i^\eps(2b)},\, N\kappa \right)\ln\big(\eps^{-1}F^0\big)\,.
    \end{equation*}
\end{enumerate}
\end{theorem}

\textbf{Collaborative speedup.} In Setting~\ref{setting1}, denoting $T_\eps(i)$ the sample complexity in the \emph{all-for-one} setting (that matched the corresponding lower bound), we observe that in the \emph{all-for-all} regime, we obtain $T_\eps\leq \frac{1}{N}\sum_i T_\eps(i)$. The speedup in comparison with a no-collaboration strategy (all agents locally performing SGD) is $\frac{1}{N}\sum_i\frac{1}{\cN_i^\eps(2b)}$: the mean of all \emph{all-for-one} speed-ups.
Similarly in Setting~\ref{setting2}, the speed-up in comparison with the no-collaboration setting is still $\frac{1}{N}\sum_i\frac{1}{\cN_i^\eps(2b)}$ in the statistical regime.
In Setting~\ref{setting2}, one could obtain $\sqrt{\kappa}$ instead of $\kappa$ in the optimization term, using accelerated gradient methods (with additive noise here), leading to a faster convergence to the statistical regime (first term of the max in Theorem~\ref{thm:afa_sample}.2), where we have the collaboration speedup.
We present in Appendix~\ref{app:time_adaptive} a time-adaptive variant with varying step sizes and matrices, that share the same sample complexity, and Theorem~\ref{thm:afa_convex} in Appendix~\ref{app:afa_convex} is the case $\mu=0$.

\textbf{Assumption~\eqref{hyp:bias1}.} Importantly, controlling the mean of local generalization errors leads to a much weaker bias assumption: instead of requiring a uniform control of gradient norms (as is done in all previous works, and in our \emph{all-for-one} setting), our \emph{all-for-all} algorithm only requires that for some $(b_{ij})$, we have $f_i(x_j^\star)-f_i(x_i^\star)\leq b_{ij}$: in the GEM setting \emph{e.g.}, even when gradients of $f_i$ and $f_j$ may differ a lot, if agent $j$'s optimal model generalizes well-enough under agent $i$'s distribution, they should collaborate. We believe this notion of function proximity that we leverage in this setting to be the weakest possible.

\textbf{Intuition behind the algorithm.} Perhaps surprisingly, matrix $W$ is in general \emph{not} a gossip matrix (\emph{i.e.}~such that $W\mathds{1}=\mathds{1}$): agent $i$ does not aggregate a convex combination of stochastic gradients, but a combination with scalars that do not necessarily sum to 1. 
In the collaborative GEM setting, we thus cannot say that the \emph{all-for-all} algorithm acts as if, in parallel, each agent $i$ trains a model on the mixture of distributions $\cD_j$ with weights $W_{ij}$. In fact, as the analysis shows below, agent $i$ trains a model on the mixture of distributions, with weights $\lambda_{ij}$, if $\Lambda$ is a stochastic square root of matrix $W$ ($\Lambda\Lambda^\top=W$). 
Thus, the \emph{all-for-all} algorithm is exactly a paralleled version of weighted gradient averagings, where agent $i$ uses the stochastic vector $(\lambda_{ij})_{1\leq j\leq N}$. 
In order to account for inter-dependencies between agents that do not directly share information, the \emph{all-for-all gradient filtering} uses weights $W_{ij}$ to aggregate information, instead of $\lambda_{ij}$.
Propagating information using a matrix $W$, that induces a similarity graph $G_W$ on $\set{1,\ldots,N}$, such that $(ij)\in E_W$ if $W_{ij}>0$, is quite natural~\citep{perso_bellet_graph,pmlr-v84-bellet18a}; yet, ours is the first analysis to give such precise generalization error bounds, through the use of a stochastic optimization framework. 

\textbf{Degrees of freedom offered by $W$.} In comparison to Theorem~\ref{thm:afa}, the classical personalized FL approaches that consider personalized local models of the form $x_i=\Bar{x}-\delta_i$, where $\Bar{x}$ is some global quantity shared by all agents, perturbed (and personalized) by some local quantity $\delta_i$ (\emph{e.g.}~averaging between local and a global models), can be seen as the special instances where, for all $i$, we have $\lambda_{ii}=1-\alpha_i$ and $\lambda_{ij}=\frac{\alpha_i}{N-1}$ if $i\ne j$ for some $\alpha_i$, and leads to bias terms of the form $\frac{1}{N}\sum_i \frac{\alpha_i}{N-1}\sum_{j\ne i}b_{ij}$. Full and naive collaboration (a single model trained for all users) corresponds to $\lambda_{ij}=1/N$ for all $i,j$, and leads to a bias term of $\frac{1}{N^2}\sum_{i,j}b_{ij}$.
The degrees of freedom offered by our matrix $W$ (and by coefficients $\lambda_{ij}$) enable pairwise agent adaptation, and tighter generalization guarantees and bias/variance tradeoffs.

\begin{proof}[Proof sketch of convergence guarantees]
Since brutally analyzing convergence of the iterates $(x^k)$ generated with $x^{k+1}=x^k-WG^k$ seems impossible due to both gradient biases and model biases between agents, we study these iterates through the introduction of a different but related problem. This approach is in fact similar to some decentralized optimization ones, where a dual problem or a related energy function is often introduced \citep{scaman2019optimal,even2021continuized}, upon which well-studied algorithms are applied.
The related problem we formulate is different from and more flexible than all the different personalized FL problems in the literature \citep{hanzely_lower_2020,t_dinh_personalized_2020}, that consider regularization terms that enforce consensus.
For $\lambda=(\lambda_{ij})_{1\leq i,j\leq N}$ a stochastic matrix (such that for all $1\leq i\leq N$, we have $\sum_{j=1}^N\lambda_{ij}=1$), let $f^\Lambda$ be defined as:
\begin{equation}\label{eq:fLambda}
    f^\Lambda(y)=\Bar{f}(\Lambda y)\,,\quad y\in\R^{N\times d}\,,
\end{equation}
where $\Bar{f}=\frac{1}{N}\sum_i f_i$. Gradient descent on $f^\Lambda$ writes as  $$y^{k+1}=y^k-\eta \Lambda^\top \nabla\Bar{f}(\Lambda y^k)\,$$
where $\nabla \Bar{f}(x) = \frac{1}{N}\big(\nabla f_i(x_i)\big)_{1\leq i\leq N}$ for any $x\in\R^{N\times d}$. Importantly, notice that denoting $x^k=\Lambda y^k$ and since $W=\Lambda\Lambda^\top$, we have the recursion $$x^{k+1}=x^k-\eta W \nabla \Bar{f}(x^k)\,$$ making an analysis of the iterates $(x^k)$ possible. In our case, we however use stochastic gradients given by our oracles. The full gradient $\nabla \Bar{f}(x)$ is thus replaced by $(G^k(x_i)_i)_i$ defined at Equation~\eqref{eq:full_stoch}. Defining $(y^k)_k$ with the recursion:
    $$y^{k+1}=y_k-\eta ((\Lambda G^k(y^k)_i)_i\,,$$
initialized at $y^0_1=x^0_1=\ldots=y^0_N=x^0_N$ we have $x^k=\Lambda y^k$, for all $k\geq0$, where $(x^k)$ is generated using Algorithm~\ref{algo:afa}. As a consequence, controlling in Theorem~\ref{thm:afa} the function values $\frac{1}{N}\sum_if_i(x_i^k)$ is equivalent to controlling $f^\Lambda(y^k)$ (these two quantities are equal).
The bias-variance trade-off thus writes as, where $x^\Lambda$ minimizes $f^\Lambda$ and $x^\star=(x_i^\star)_{1\leq i\leq N}$:
\begin{equation*}
    F^k\leq \underbrace{f^\Lambda(x^\Lambda)-\Bar{f}(x^\star)}_\text{\emph{Bias term}} + \underbrace{f^\Lambda(y^k) - f^\Lambda(x^\Lambda)}_\text{\emph{Optimization and variance terms}}\,.
\end{equation*}
\end{proof}


\section{Collaborative GEM\label{sec:gem}}

In this section, we place ourselves in the GEM setting (loss function $\ell$, distributions $(\cD_i)_i$). We briefly elaborate (see Appendix~\ref{sec:discuss} for further details) on the collaborative GEM setting we used through the paper, clearly define the distribution-based distances related to bias assumptions~\eqref{hyp:bias1} and~\eqref{hyp:bias3}, and introduce toy problems under which the bias upperbounds may be known.

\subsection{Distribution-based distances.} For $\cH$ a set of functions from $\Xi$ to $\R^d$ and $\cD,\cD'$ two probability distributions on $\Xi$, we define:
\begin{equation*}
	d_\cH(\cD,\cD')=\sup_{h\in\cH}\NRM{\esp{h(\xi)-h(\xi')}}\,,
\end{equation*}
where $\xi\sim\cD$ and $\xi'\sim\cD'$. $d_\cH$ is a pseudo-distance on the set of probability measures on $\Xi$. The definition of $d_\cH$ is motivated by the fact that, for fixed $x\in\cX$ and $1\leq i,j\leq N$, we have:
\begin{equation*}
		\NRM{\nabla f_i(x)-\nabla f_j(x)}\leq d_\cH(\cD_i,\cD_j)\,,
\end{equation*}
if $\big(\xi\in\Xi\mapsto\nabla_x\ell(x,\xi)\big) \in \cH$ where $(\xi_i,\xi_j)\sim\cD_i\times \cD_j$. Thus, if for all $1\leq i,j\leq N$, $d_\cH(\cD_i,\cD_j)\leq \hat{b}_{ij}$ for some weights $\hat{b}_{ij}$,
bias assumption~\eqref{hyp:bias1} holds for $b_{ij}=\sup_{k,l}\NRM{x_k^\star-x_l^\star}\hat{b}_{ij}$ (and for $\hat{b}_{ij}^2/(2\mu)$ under $\mu$-PL assumption), and bias assumption~\eqref{hyp:bias2} holds for $\tilde{b}_{ij}=\hat{b}_{ij}^2$.

\subsection{Weakly supervised setting.} In a scenario where, from a large pool of unlabelled data (distributions $(\hat{\cD}_i)_i$), requiring a label led to a sample from $\cD_i$ but were costly, agents would benefit from computing distance-based distributions between unlabelled distributions $(\hat{\cD}_i)_i$ (possible thanks to a large amount of unlabelled samples); these computed distances would then help to reduce the number of labelled samples required.

\subsection{Geometric structure of the agents and infinitely-many-agents.} A toy problem is to assume that agent distributions are drawn in an \emph{i.i.d.}~fashion from a continuous set of possible distributions $\Theta\subset\R^p$:  $\theta_1,\ldots,\theta_N\sim\nu$ where $\nu$ is a density over the set $\Theta$. Making $N\to\infty$, we end up with a continuum of agents of distributions $(\cD_\theta)_{\theta\in\Theta}$ with density $\nu$. In that setting, under a bias assumption of the form $d_\cH(\cD_\theta,\cD_{\theta'})\leq \hat{b}(\theta,\theta')\leq \NRM{\theta-\theta'}^q$ for all $\theta,\theta'$, the collaboration speedup of Theorem~\ref{thm:afa_sample}.1 (for instance), writes as:
\begin{equation*}
            \int_\Theta \big(f_\theta(x_\theta)-f_\theta(x_\theta^\star)\big)\nu(\dd\theta)\leq \eps\,,
        \end{equation*} with a total number of samples drawn of
        \begin{equation*}
            T_\eps \leq \cO_{\eps\to0}\left(B^2D^2\eps^{-2-p/q}\right)\,,
        \end{equation*}
under regularity assumptions on $\nu$ and $\hat{b}$. This approach is made more rigorous in Appendix~\ref{sec:infinitely}.

\subsection{Illustration of our theory.}
\begin{figure}[b!]
\centering\vspace{-9pt}
\hfill
\subfigure[Comparison of \emph{All-for-all} with naive benchmarks]{
    \includegraphics[width=0.45\linewidth]{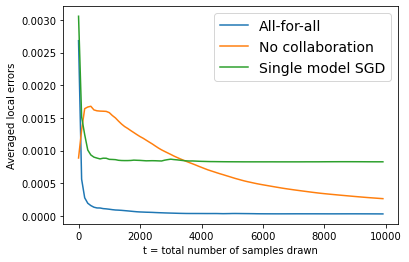}
    \label{fig:afa_comparison}
}
\hfill
\subfigure[Effect of noise in the estimation of bias parameters $b_{ij}$]{
    \includegraphics[width=0.45\linewidth]{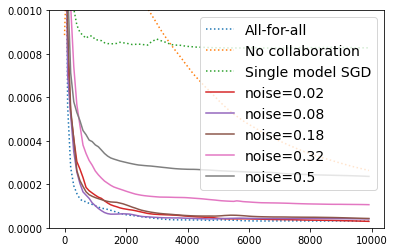}
    \label{fig:afa_noisy}
}
\hfill
\vspace{-14pt}
\caption{\textbf{All-for-all algorithm in practice}
\label{fig:afa}
}
\vspace{-9pt}
\end{figure}
To test the robustness of our algorithms and our theory, we use in our experiments the mean estimation problem used in the strongly convex and smooth lower bound of Theorem~\ref{thm:lower_afo}: $\ell(x,\xi)=\frac{1}{2}\NRM{x-\xi}^2$, $x\in\R^d$ for $\xi$ a $d$-dimensional Bernoulli random variable. We use the time-adaptive version of Algorithm~\ref{algo:afa}, and we place ourselves in the setting of Section~\ref{sec:infinitely}, where agent distributions are drawn from a distribution of distributions: for $N=100$ agents, we draw $(p_i)_{1\leq i\leq N}$ \emph{i.i.d.}~uniformly distributed in $[0,1]$, and $p_i$ is the parameter of agent $i$'s Bernoulli variables. We draw $10^3$ samples for each agent, and we compute for $1\leq t\leq T$ the averaged local generalization errors (here the error from the mean), namely $\frac{1}{N}\sum_{i=1}^N \frac{1}{2}\NRM{x_i^t-p_i}^2$, where $x_i^t$ is the output of the algorithm after $t$ samples drawn.
As expected and as illustrated in Figure~\ref{fig:afa_comparison}, our \emph{all-for-all} approach benefits from both \emph{no-collaboration} (each agent locally estimating its mean $p_i$ with only its locally available samples, in orange in the graph) and \emph{single-model} approaches (a fully centralized minimization of $\frac{1}{N}\sum_i f_i(x)$, in green in the graph), through both a convergence to the true mean, and a non-asymptotic acceleration.
In Figure~\ref{fig:afa_noisy}, we study the effect of noise on the estimation of bias parameters $b_{ij}$ (that here correspond to $b_{ij}=\frac{1}{2}(p_i-p_j)^2$), by taking as inputs in algorithms $b_{ij}^{\rm noisy}=\frac{1}{2}(p_i+n_i-p_j-n_j)^2$, where $n_i$ are \emph{i.i.d.}~uniformly distributed in $[-{\rm noise},{\rm noise}]$, for different noise values (Figure~\ref{fig:afa_noisy}). \emph{All-for-all} algorithm appears to be quite robust to noise: for small noise amplitudes (0.02 and 0.08, and even 0.18), performances are not too degraded. For (abusively) large noise values (0.32 and 0.5), the non-asymptotic speedup is kept, with degraded asymptotic performances. Still, in the range of parameters considered, these perform better than no-collaboration and single model approaches.

\section*{Discussion and conclusion.}
In this paper, we quantified in term of function biases ($b_{ij}$), stochastic gradient noise or amplitudes ($\sigma^2$ or $B^2$), target precision $\eps>0$ and functions regularity parameters, the benefit of collaboration between agents for shared minimization using stochastic gradient algorithms. 
Our lower bound (Theorem~\ref{thm:lower_afo}) states that in the~\emph{all-for-one} setting, assumption parameters being fixed, no algorithm matches the performances of weighted gradient averagings in terms of sample optimality. 
Another lower bound (and corresponding upper-bounds) that would be worth investigating is: without any knowledge on biases $b_{ij}$, what is the worst case complexity of an algorithm that would thus need to \emph{learn who to learn with}? 
We extended weighted gradient averagings to the \emph{all-for-all} setting, technically more challenging, through the introduction of a related problem, simplifying the analysis of such algorithms. 

The main limitation of our work lies in the assumption that upper-bounds on the biases $b_{ij}$ are known. We investigate in Section~\ref{sec:discuss} scenarii in which this assumption is valid. Yet, as mentionned above, \emph{learning who to learn with and how} is a challenging question worth tackling. Extending our work to model agnosticism together with data heterogeneity would require bias assumptions such as our distribution-based ones in Section~\ref{sec:discuss}. Finally, fairness issues might be raised by our approaches in the \emph{all-for-all} setting: we consider the average errors amongst agents and therefore do not ensure bounds on the supremum of all local errors. Still, incentives to collaborate and send gradients to other users are the hope of being helped back (matrix $W$ is symmetric in Algorithm~\ref{algo:afa}).

\bibliography{biblio}
\bibliographystyle{plainnat}

\onecolumn
\appendix

\section{The Case of Collaborative Generalization Error Minimization (GEM) \label{sec:discuss}}

In most of the related works and in this paper, the quantities equivalent to our bias assumptions are assumed to be known by the optimizer. Yet, while this assumption makes analyses possible, the knowledge of these quantities is a strong assumption. In this section, we provide settings and toy problems under which these assumptions are natural.
We first begin by introducing distribution based distances between agents, in order to control gradient biases in the setting where functions are of the form~\eqref{eq:function_loss} for a shared loss $\ell$ and local distributions $(\cD_i)_i$, for collaborative GEM. We then present two settings for this problem.

\subsection{Distribution-based distances}

In order to quantify the bias/variance tradeoff that appears when using stochastic gradients computed by other agents, we define the following notion of distance.
\begin{definition}\label{def:dH}
For $\cH$ a set of functions from $\Xi$ to $\R^d$ and $\cD,\cD'$ two probability distributions on $\Xi$, we define:
\begin{equation*}
	d_\cH(\cD,\cD')=\sup_{h\in\cH}\NRM{\esp{h(\xi)-h(\xi')}}\,,
\end{equation*}
where $\xi\sim\cD$ and $\xi'\sim\cD'$. $d_\cH$ is a pseudo-distance on the set of probability measures on $\Xi$.
\end{definition}
Some instances of these distances encompass:
\begin{enumerate}
    \item For $\cH=\cH_{\rm Affine}$ the set of 1-Lipschitz affine functions, we have $d_\cH(\cD,\cD')=\NRM{\esp{\xi}-\esp{\xi'}}$.
    \item For $\cH=\cH_{\rm Lipschitz}$ the set of 1-Lipschitz functions on $\Xi$, we have $d_\cH(\cD,\cD')=\cW_1(\cD,\cD')$ the 1-Wasserstein distance (or EMD, earth mover distance).
    \item For $\cH=\cH_{\rm Bounded}$ the set of of functions whose values lie in a space of diameter 1, we have $d_\cH(\cD,\cD')=d_{\rm TV}(\cD,\cD')$ the total variation distance.
    \item For $\cH=\cH_{\rm Loc. Bounded}$ defined as $\set{h:\Xi\to\R^d\quad \text{s.t.}\quad \forall \xi\in\Xi,\, {\rm Diam}\big(h(\cB(\xi,1)\cap\Xi)\big) \leq 1 }$, we also obtain the total variation distance.
\end{enumerate}
The definition of $d_\cH$ is motivated by the fact that, for fixed $x\in\cX$ and $1\leq i,j\leq N$, we have:
\begin{equation*}\label{eq:nabla_H}
	\begin{aligned}
		\NRM{\nabla f_i(x)-\nabla f_j(x)}&=\NRM{\esp{\nabla_x\ell(x,\xi_i)-\nabla_x\ell(x,\xi_j)}}\\
		&\leq d_\cH(\cD_i,\cD_j)\,.
	\end{aligned}
\end{equation*}
if $\big(\xi\in\Xi\mapsto\nabla_x\ell(x,\xi)\big) \in \cH$ where $(\xi_i,\xi_j)\sim\cD_i\times \cD_j$.

We thus add a fourth bias assumptions to our list in Section~\ref{sec:hyp}, called distribution-based bias assumption: for some function set $\cH$, for all $1\leq i,j\leq N$, for all $x\in\R^d$, $d_\cH(\cD_i,\cD_j)\leq \hat{b}_{ij}$ for some non negative weights $\hat{b}$, and for all $x\in\R^d$, $\big(\xi\in\Xi\mapsto\nabla_x\ell(x,\xi)\big) \in \cH$.
This assumptions is verified in the following settings for mean estimation and regression.
\begin{enumerate}
    \item For $\ell(x,\xi)=\frac{1}{2}(x-\xi)^2$ (mean estimation), distribution-based bias assumption holds for $\cH_{\rm Affine}$, and the distance to consider is thus $\NRM{\esp{\xi_i}-\esp{\xi_j}}$.
    \item For quadratic loss (linear regression) $\ell(x,\xi)=\frac{1}{2}(a^\top x-b)^2$ where $\xi=(aa^\top, ba^\top)$, we have $\cH= D \cH_{\rm Affine}$, where $D$ is the diameter of the space in which are iterates lie, and the distance to consider is thus a scaled 1-Wassertein distance between distributions 
    \item For logisic regression $\ell(x,\xi)=\log(1+e^{-b a^\top x})$ where $\xi=ba^\top$, we have $\cH=(1+\frac{1}{e})\cH_{\rm Lipschitz}$ and the distance to consider is thus the 1-Wassertein distance between distributions (scaled with a constant factor).
    \item For the hinge loss $\ell(x,\xi)=\max(0,1-b a^\top x)$, for $\xi=ba^\top$ we have $\cH=D\cH_{\rm Bounded}$ where $D$ is the diameter of the space in which are iterates lie, and the distance to consider between distributions is thus a scaled total variation distance.
\end{enumerate}

\noindent Thus, instead of controlling uniform concentration bounds on the gradients (under the GEM assumption, this can be done using anisotropic uniform concentration bounds, \emph{e.g.}~\citet{even2021tensors}), one can leverage our bias assumptions by controlling distribution-based distances.

\begin{proposition}
    If for all $1\leq i,j\leq N$, $d_\cH(\cD_i,\cD_j)\leq \hat{b}_{ij}$ (\emph{i.e.}~distribution-based bias assumption holds), then:
    \begin{enumerate}
        \item Bias assumption~\eqref{hyp:bias1} holds for $b_{ij}=D^\star\hat{b}_{ij}$;
        \item Bias assumption~\eqref{hyp:bias2} holds for $\tilde{b}_{ij}=\hat{b}_{ij}^2$.
    \end{enumerate}
\end{proposition}

\subsection{Lower-bounds (Theorem~\ref{thm:lower_afo}) in term of distribution based distances}

Theorem~\ref{thm:lower_afo} being formulated as IT-lower bounds for collaborative GEM, the distribution-based distances we introduced perfectly adapt to it. In fact, in the proof of the lower bounds, the distributions we build verify $d_\cH(\cD_i,\cD_j)\leq b_{ij}$ if $\cH \subset \cH_{Loc. Bounded}$, so that the following directly holds.
\begin{theorem}\label{thm:lower_distrib}Let $\eps\in(0,1/16)$. 
	Assume that the function set $\cH$ satisfies:
\begin{equation*}
	\cH_{\rm Affine} \subset \cH \subset \cH_{Loc. Bounded}\,,
\end{equation*}
	and assume that either the synchronous or asynchronous oracle is used.
	Assume that $(b_{ij})$ verifies the triangle inequality $b_{ij}\leq b_{ik}+b_{kj}$ for all $1\leq i,j,k\leq N$. 
	Let $\Tilde{\D}(r,b,B)$ be the set $\D(r,b,B)$, but where the bias assumptions are replaced by the distribution-based bias assumption ($d_\cH(\cD_i,\cD_j)\leq b_{ij}$), and similarly for $\Tilde{\D}_{\mu=1/r^2}^{L=1/r^2}(r,b,\sigma)$.
	We have, for some constant $C>0$ independent of the problem and any $i\in\{1,\ldots,N\}$:
	\begin{equation*}
		\inf_{\cM\in\M}\!\sup_{((\cD_j)_j,\ell)\in\Tilde{\D}(r,b,B)}\!\!\!\!\!\! \cT_i^{\eps}\Big(\cM,\big((\cD_j)_j,\ell\big)\Big) \!\geq\! \frac{CB^2r^2N}{\eps^{2}\cN_i^\eps(\frac{b}{4})}\,,
	\end{equation*}
	where $\cN_i^\eps(b)=\sum_j \mathds{1}_{\{b_{ij}\leq \eps\}}$ is the number of agents $j$ verifying $b_{ij}\leq \eps$.
	Under strong-convexity and smoothness assumptions for $\mu=L=1/r^2$, this bound becomes, for some constant $C'>0$:
	\begin{equation*}
		\inf_{\cM\in\M}\!\sup_{((\cD_j)_j,\ell)\in\Tilde{\D}_{\mu=1/r^2}^{L=1/r^2}(r,b,\sigma)}\!\!\!\!\!\! \cT_i^{\eps}\Big(\cM,\big((\cD_j)_j,\ell\big)\Big) \geq \frac{C'r^2\sigma^2N}{\eps\cN_i^\eps(\frac{b}{4})}\,.
	\end{equation*}
\end{theorem}

\subsection{Weak supervision}

We explicit in this section another paradigm in which the knowledge of some distribution-based bias assumptions is realistic, in a weakly supervised learning setting. More precisely, assume that the desired task is a classification or regression one: random variables $\xi$ are of the form $\xi=(\zeta,\beta)$ for some $\zeta\in Z\subset\R^{D_1}$ and some \emph{label} $\beta\in\R^{D_2}$.
For every agent $1\leq i\leq N$, let $\cD_i^\zeta$ be the $\zeta$-marginal of $\cD_i$.
In our weakly supervised learning setting, we assume that there exists some (eventually random) function $\cL$ that maps unlabelled data $\zeta$ to their labelled ones \emph{i.e.}~for some $\cL:Z\to\Xi$, the random variable $\cL(\zeta_i)$ is of same law as $\xi_i\sim\cD_i$ for $\zeta_i\sim\cD_i^\zeta$.
While many samples may be accessible and drawn from $\cD_i^\zeta$ for every agent $1\leq i\leq N$, labelling data \emph{i.e.}~applying function $\cL$ in order to recover samples of law $\cD_i$ is assumed to be costly. For instance, applying $\cL$ in medical data analysis could require the help form an expert (a radiologist, \ldots).

Agents would benefit from computing distances between distributions $\cD_i$ using marginals $\cD_i^\zeta$: under mild assumptions on the function $\cL$ (Lipschitzness with high probability, mainly), distances $d_\cH(\cD_i,\cD_j)$ are upperbounded by distances between $\cD_i^\zeta$ and $\cD_j^\zeta$.
Thus, in this setting, agents would benefit from collaborating: computing distribution-based distances would not require expert advices, while reducing the total number of expert calls required by reducing the number of samples needed to be drawn from each $\cD_i$ through collaboration.

\subsection{Infinitely-Many-Agents Limit and Geometrical Prior Knowledge\label{sec:infinitely}}

We introduce the \emph{infinitely-many-agents} limit, by taking $N\to\infty$ with added structure on the problem below. This toy-problem illustrates a geometric structure under which the knowledge of distribution-based bias assumptions is realistic.
Let $\cP$ be a set of probability laws and $\Theta\subset\R^p$ be a set parameterizing $\cP$ in the sense that $\cP$ can be written as $\{\cD_\theta,\,\theta\in\Theta\}$ where $\cD_\theta,\,\theta\in\Theta$ are probability laws on $\Xi$.
For all $1\leq i\leq N$, we assume that there exists $\theta_1,\ldots,\theta_N\in\Theta$ such that $\cD_i=\cD_{\theta_i}$, and that the agents are drawn from a probability distribution with density $\nu$ on $\Theta$: $(\theta_i)_i$ is a sequence of \emph{i.i.d.}~vectors sampled from $\Theta$ with density $\nu$.

This aims at modelling a \emph{geometric} structure on the agents distributions: agents $i$ and $j$ such that $\theta_i$ and $\theta_j$ are nearby in the set $\Theta$ share similar distributions. 
We quantify this in the following way: there exists a continuous function $b:\Theta^2\to \R^+$ such that for any $(\theta,\theta')\in\Theta^2$, $d_\cH(\theta,\theta')\leq b(\theta,\theta')$.
In sensor networks, $\Theta$ is some domain of $\R^2$ of $\R^3$, each agent is a sensor located at a physical position in $\Theta$, whose goal is to predict based on noisy local observations (\emph{e.g.}~temperature, pollution,~\ldots) that are space-continuous. 
More generally, if agents model a population (agent $i$ being an individual in the population) in some geographical location $\Theta\subset\mathcal{S}_2$ (sphere in $\R^3$) (ranging from cities, countries, continents, or the whole wide world), and if we are interested in speech-recognition models, a fair assumption is to have similar local distributions for nearby agents. The distribution $\nu$ is then the geographical density of agents.

The \emph{infinitely-many-agents} limit is obtained by taking $N\to\infty$ with our geometric structure on the agents: we obtain a continuum of agents parameterized by $\theta\in\Theta$, with density $\nu$ on $\Theta$. For every $\theta\in\Theta$, we write:
\begin{equation*}
	f_\theta(x)=\E_{\xi\sim \cD_\theta}[\ell(x,\xi)]\,,\quad x\in\cX\,,
\end{equation*}
and we denote by $x_\theta^\star$ a minimizer over $\cX$ of $f_\theta$. 
The algorithm and model we consider are defined as follows, inspired by Oracle 2.
\vspace{5pt}
\begin{center}
	\fbox{
	\begin{minipage}{8cm}
	\begin{center}{\textbf{Infinitely-many-agents oracle and algorithm}}\end{center}

	At each time-step $t=1,2,...$:
	\begin{description}
		\item[1] An agent parameterized by $\theta_t\sim\nu$ is ‘awakened', draws $\xi_{\theta_t}^t$ is drawn from data distribution $\cD_{\theta_t}$ and computes a stochastic gradient $g_t=\nabla_x \ell(y_t,\xi_{\theta_t}^t)$. 
		\item[2] Agents $\theta\in\Theta$ such that $w(\theta,\theta_t)\leq s_\eps$ for some fixed $s_\eps$ receive $g_t$.
		\item[3] Upon reception of $g_t$, agents may update their local estimates.
	\end{description}
	\end{minipage}
	}
	\end{center}
At finite time-horizon $T>0$, only a finite number of agents may have ‘awakened' and drawn (at most) one sample from their local distribution; yet, we aim at obtaining upper-bounds on the local generalization errors for all agents.
Applying our AFO and AFA algorithms, under aditional assumptions, yields the following bounds.
\begin{proposition}
    Assume that $\nu$ is continuous on $\Theta$, and that for all $\theta,\theta'\in\Theta$, we have $b(\theta,\theta')\leq \NRM{\theta-\theta'}^{q}$. Assume that assumption~\ref{hyp:noise2} holds and that iterates lie in a space of diameter $D$.
    \begin{enumerate}
        \item \emph{All-for-one:} let $\eps>0$ and $\theta\in\interior{A}$. There exists a procedure such that $T_\eps(\theta)$ the time needed to reach a precision $\eps$ for agent $\theta$ is upper-bounded by: \begin{equation*}
            T_\eps(\theta)\leq \cO_{\eps\to0}\left(B^2D^2\nu(\theta)\eps^{-2-p/q}\right)\,.
        \end{equation*}
        \item \emph{All-for-all:} for any $\eps>0$, there exists a procedure making $T_\eps$ oracle calls and answering $(x_\theta)_\theta\in\cX^\Theta$ such that:\begin{equation*}
            \int_\Theta \big(f_\theta(x_\theta)-f_\theta(x_\theta^\star)\big)\nu(\dd\theta)\leq \eps\,,
        \end{equation*} with \begin{equation*}
            T_\eps \leq \cO_{\eps\to0}\left(B^2D^2\eps^{-2-p/q}\right)\,.
        \end{equation*}
    \end{enumerate}
\end{proposition}

\begin{proof}
For finite number of agents $N$, we have $\theta_1,\ldots,\theta_N$ drawn \emph{i.i.d.} from $\nu$. Using weighted gradient averagings and Theorem~\ref{thm:wga_sample1}, we have for agent $i$, and denoting $\theta=\theta_i$:
\begin{equation*}
    T_\eps(\theta)\leq 2B^2D^2\eps^{-2} \frac{\cN_i^\eps(2b)}{N}\,.
\end{equation*}
Here, $\cN_i^\eps(2b)\leq |\set{1\leq j\leq N\,:\,\NRM{\theta-\theta_j}\leq (\eps/2)^{1/q}}=\sum_{j=1}^N \mathds{1}_{\set{\NRM{\theta-\theta_j}\leq (\eps/2)^{1/q}}}$. Using the strong law of large number, we have almost surely, as $N\to\infty$:
\begin{equation*}
    \frac{\cN_i^\eps(2b)}{N}\leq \int_{\cB(\theta,(\eps/2)^{1/q})} \nu(z)\dd z +o(1)\,.
\end{equation*}
Thus, for $N\to\infty$:
\begin{equation*}
    T_\eps(\theta)\leq 2B^2D^2 \int_{\cB(\theta,(\eps/2)^{1/q})} \nu(z)\dd z\,.
\end{equation*}
Then, making $\eps\to0$, we have $\int_{\cB(\theta,(\eps/2)^{1/q})} \nu(z)\dd z\sim_{\eps\to0} \nu(\theta){\rm Vol}(\cB(\theta,(\eps/2)^{1/q}))=\cO\left(\nu(\theta)\eps^{-p/q}\right)$ ($\nu$ is continuous). We thus have the result for $\theta$ such that $\nu(\theta)>0$ in the \emph{all-for-one} setting.

In the \emph{all-for-all} setting, we use Theorem~\ref{thm:afa_sample}.1 in for a finite number of agents, and we similarly use the strong law of large numbers as $N\to\infty$.
\end{proof}

\section{Proof of Lower-Bounds (Theorem~\ref{thm:lower_afo} and Corollary~\ref{cor:lower_afa}) \label{app:lower}}

\subsection{General framework to prove lower bounds}
The idea is that, when optimizing a function $f(x)=\esp{\ell(x,\xi)}$ and finding a good approximation of a minimizer $x^\star$, we learn some information on the distribution $\cD$ over which samples are drawn. 
In order to prove lower bounds, we construct a loss function $\ell$, and distributions $\cD_1^\alpha,\ldots,\cD_N^\alpha$, where $\alpha$ is a random parameter. We argue that minimizing (in the \emph{all-for-all} or \emph{all-for-one} settings) the objective function up to a certain precision gives a good estimator (quantified) of the random seeds $\alpha$. Then, using Fano inequality, we bound the efficiency of such an estimator in terms of number of oracle calls, obtaining a lower bound on the sample complexity. This approach is inspired by~\citet{agarwal2012ITlowerbounds}, who prove IT-lower bounds for stochastic gradient descent. We adapt their proof technique to the personalized and multi-agent setting.

\subsubsection*{Constructing difficult loss functions} 

For any two functions $f,g:\R^d\to\R$, we define the discrepancy measure $\rho(f,g)$ as:
\begin{equation*}
    \rho(f,g)=\inf_{x\in\R^d}\Big\{f(x)+g(x)-f(x_f^*)-g(x_f^*)\Big\}\,,
\end{equation*}
which is a pseudo metrics. Now, for a finite set $\cV$ of parameters, let $\cG(\delta)=\set{g_\alpha^\delta\,,\,\alpha\in\cV}$ be a set of functions indexed by $\cV$, that depend on $\delta$ (fixed in the set). The dependency of each $g_\alpha\in\cG(\delta)$ is left implicit in the following subsections. We define:
\begin{equation*}
    \psi(\delta)=\inf_{f,g\in\cG(\delta)}\rho(f,g)\,.
\end{equation*}

\subsubsection*{Minimizing is Bernoulli parameters identification} 
The two following lemmas justify that optimizing a function $g_\alpha\in\cG(\delta)$ to a precision of order $\psi(\delta)$ is more difficult than estimating the parameter $\alpha$.
\begin{lemma}[\citet{agarwal2012ITlowerbounds}]\label{lem:precision_psi}
    For any $x\in\R^d$, there can be at most one function $g_\alpha$ in $\cG(\delta)$ such that:
    \begin{equation*}
        g_\alpha(x)-\inf_{\R^d} g_\alpha < \frac{\psi(\delta)}{3}\,.
    \end{equation*}
    \end{lemma}
\begin{lemma}[\citet{agarwal2012ITlowerbounds}]\label{lem:estimate_alpha}
    Assume that for some fixed but unknown $\alpha\in\cV$ there exists a method $\cM_T$ based on the data $\phi=\{X_1,...,X_T\}$ that returns $x^T$ (function of $\phi$) satisfying an error of:
    \begin{equation*}
        \esp{g_\alpha(x^T)-\min_{x\in\R^d}g_\alpha(x)} < \frac{\psi(\delta)}{9}\,,
    \end{equation*}
    where the mean is taken over the randomness of both the oracle $\Phi$, the method $\cM_T$ and $\alpha\in\cV$ if random.
    Then, there exists a hypothesis test $\Hat{\alpha}:\phi\to\cV$ such that:
    \begin{equation*}
        \max_{\alpha\in\cV}\P_\phi\big(\Hat{\alpha}\ne\alpha\big)\leq \frac{1}{3}\,.
    \end{equation*}
\end{lemma}
Suppose now that the parameter $\alpha$ in the previous Lemma is chosen uniformly at random in $\cV$. Let $\Hat{\alpha}:\phi\to\cV$ be a hypothesis test estimating $\alpha$. By Fano inequality \citep{information_theory}, we have:
\begin{equation}\label{eq:fano}
    \P\big(\Hat{\alpha}\ne\alpha)\geq 1-\frac{I\big(\phi,\alpha\big)+\ln(2)}{|\cV|}\,,
\end{equation}
where $I\big(\phi,\alpha\big)$ is the mutual information between $\phi$ and $\alpha$, that we need to upper-bound.
Combining Fano inequality with Lemmas~\ref{lem:precision_psi} and~\ref{lem:estimate_alpha}, fixing a target error $\eps=\psi(\delta)$, we obtain a lower bound on the sample complexity $T_\eps$:
\begin{equation*}
    \frac{1}{3}\geq \P_\phi\big(\Hat{\alpha}\ne\alpha\big) \geq 1-\frac{I\big(\phi_{T_\eps},\alpha\big)+\ln(2)}{|\cV|}\,,
\end{equation*}
where $\phi_{T_\eps}$ is the information contained in $T_\eps$ oracle calls. If we have an equality of the form $I\big(\phi_{T_\eps},\alpha\big)=T_\eps I\big(\phi_{1},\alpha\big)$, this gives:
\begin{equation}\label{eq:lower_gen}
    T_\eps \geq \frac{\frac{2}{3}|\cV|-\ln(2)}{I(\phi_1,\alpha)}\,.
\end{equation}
Playing with the different parameters $\delta,\alpha,\cV$ gives lower bounds in our \emph{all-for-one} settings. We refer the interested reader to Chapter 2 in~\citet{information_theory} for Fano inequality and mutual information.

\subsection{Applying this in the \emph{all-for-one} setting, first part of Theorem~\ref{thm:lower_afo}}

We first prove the lower bounds in the case of the asynchronous oracle (one data item sampled from personal distribution at each iteration). 

We define the loss $\ell:\R^d\times \Xi$ where $\Xi\subset \R^d$ and $\delta>0$:
\begin{equation*}
    \forall(x,\xi)\in\cX\times\Xi\,,\quad \ell(x,\xi)=\frac{1}{d}\sum_{i=1}^d\Big[ \xi f_i^+(x)+(1-\xi)f_i^-(x)-\delta x_i\Big]\,,
\end{equation*}
where 
\begin{equation*}
    f_i^+(x)=|x_i+1/2|\,\quad f_i^-(x)=|x_i-1/2|\,.
\end{equation*}
Now, for any $\alpha\in\{-1,1\}^d$, $\delta\in[0,1/2]$ being fixed, we define:
\begin{equation}\label{eq:difficult_function}
    g_\alpha(x)=\frac{1}{d}\sum_{i=1}^d\Big[ \big(\frac{1}{2}+\alpha_i\delta\big) f_i^+(x)+\big(\frac{1}{2}-\alpha_i\delta\big)f_i^-(x)-\delta x_i\Big]\,.
\end{equation}

For simplicity, assume that $r^2=d$ and $B^2=1$.
Let $\delta>0$ a free parameter. Let $\cV=\{\alpha^1,\ldots,\alpha^K\}\subset \{-1,1\}^d$ be a subset of the hypercube such that for all $k\ne l$,
\begin{equation*}
    \frac{1}{2}\sum_{i=1}^d |\alpha_i^k-\alpha_i^l|\geq \frac{d}{4}\,,
\end{equation*}
\emph{i.e.}~$\cV$ is a $d/4$-packing of the hypercube. We assume that $-\one$ ($d$-dimensional vector with $-1$ at all its entries) is not in $\cV$ and that for all $\alpha\in\cV$, $\frac{1}{2}\sum_{i=1}^d |\alpha_i+1|\geq \frac{d}{4}$.  We know that we can set $|\cV|\geq(2/\sqrt{e})^{d/2}$. 
Without loss of generality, we prove a lower bound in the case where the agent that desires to minimize its local function is indexed by $1$. 
For any $i=1,\ldots,N$, let $\cD_i$ be the probability distribution on $\{0,1\}^d$ of the following random variable:
\begin{equation*}
        {\rm Ber}\big(\frac{1}{2}+\delta_i \alpha_k\big) \epsilon_k \quad \text{where}\quad \delta_i=(\delta-b_{i1})^+\,,\\
\end{equation*}
where $s^+=\max(0,s)$ for $s\in\R$, $k$ is taken uniformly at random in $\{1,\ldots,d\}$, $(\epsilon_k)$ is the canonical basis of $\R^d$, and ${Ber}(p)$ is a Bernoulli random variable, independent of $k$.

The function $f_1(x)=\esp{\ell(x,\xi_1)}$ for $\xi_1\sim\cD_1$ verifies $f_1=g_\alpha$ for $g_\alpha$ defined in~\eqref{eq:difficult_function}. For our fixed $\delta>0$ and any $\alpha,\beta\in\cV$, $g_\alpha$ is minimized at $x^\alpha=-\alpha/2$ and we have:
\begin{equation*}
    g_\alpha(x^\alpha)=\frac{1}{d}\sum_{k=1}^d\frac{1-\delta\alpha_k}{2}\,.
\end{equation*}
The discrepancies thus write as:
\begin{equation*}
    \rho(g_\alpha,g_\beta)=\frac{\delta}{d}\sum_{k=1}^d|\alpha_k-\beta_k|\,,
\end{equation*}
so that $\psi(\delta)\geq \delta/4$. Each $g_\alpha$ is minimized for $x^\alpha=-\alpha/2\in\cB_\infty(0,1/2)\subset\cB(0,\sqrt{d}/2)\subset\cB(x^0,r)$ for $x^0=0$ (where $\cB_\infty$ denotes a ball for the infinity norm): we are in the case $r=\sqrt{d}$.

We now bound the quantities $f_i(x_j^\star)-f_i(x_i^\star)$. For any $i=1,\ldots,N$, we have $x_i^\star=x^\alpha=-\alpha/2$ if $\delta_i\geq \delta/2$, and $x_i^\star=\one/2$ otherwise. For $1\leq i,j\leq N$, if both $\delta_i\geq\delta/2$ and $\delta_j\geq \delta/2$, then $x_i^\star=x_j^\star=x^\alpha$, and $f_i(x_j^\star)-f_i(x_i^\star)=0$. Similarly, if both $\delta_i<\delta/2$ and $\delta_j< \delta/2$, $x_i^\star=x_j^\star=\one/2$ and $f_i(x_j^\star)-f_i(x_i^\star)=0$. If $\delta_i\geq\delta/2$ and $\delta_j< \delta/2$, then
\begin{align*}
    f_i(x_j^\star)-f_i(x_i^\star)&=f_i(\one/2)-f_i(x^\alpha)\\
    &=\frac{1}{d}\sum_{k=1}^d \frac{|\alpha_k-1|}{2}(2\delta_i-\delta)\\
    &\leq 2\delta_i-\delta\\
    &\leq 2(\delta_i-\delta_j)\,.
\end{align*}
Since $\delta_i-\delta_j\leq b_{ij}/2$ using 1-Lipschitzness of the positive value and the triangle inequality verified by the vector $b$, we have $f_i(x_j^\star)-f_i(x_i^\star)\leq b_{ij}$. If $\delta_i<\delta/2$ and $\delta_j> \delta/2$, we similarly have $f_i(x_j^\star)-f_i(x_i^\star)\leq \frac{1}{d}\sum_{k=1}^d \frac{|\alpha_k-1|}{2}(\delta-2\delta_i)\leq b_{ij}$.

Since the functions $f_i$ are differentiable at all points where $x_k\ne1/2,-1/2$, bounding $\NRM{\nabla f_i(x)-\nabla f_j(x)}^2$ at the points where $f_i$ and $f_j$ are differentiable is enough to have~\eqref{hyp:bias2}. For $x_k<-1/2$, we have $\nabla_kf_i(x)=\nabla_kf_j(x)=-1-\delta$ and $\nabla_kf_i(x)=\nabla_kf_j(x)=1-\delta$ for $x_k>1/2$. For $-1/2<x_k<1/2$, we have $\nabla_kf_i(x)-\nabla_kf_j(x) -2(\delta_i-\delta_j)\alpha_k $, so that $\NRM{\nabla f_i(x)-\nabla f_j(x)}^2\leq 4(\delta_i-\delta_j)^2\leq b_{ij}^2/d$ at all points where $f_i$ and $f_j$ are differentiable. The gradients scale with $r$, hence the dependency in $r$ (we recall that $r^2=d$ in our case).

The mutual information $I(\phi_T,\alpha)$ writes as:
\begin{equation*}
    \begin{aligned}
        I(\phi_T,\alpha)&=T I(\phi_1,\alpha)\\
        &=\frac{T}{dN}\sum_{k=1}^d\sum_{i=1}^NI\left({\rm Ber}\big(\frac{1}{2}+\mathds{1}_{b_{1i}\leq\delta}\alpha_k\delta_i\big),\alpha_k\right)\\
        &\leq\frac{C_1T}{N}\sum_{i=1}^N \mathds{1}_{b_{1i}<\delta}\delta_i^2\\
        &=\frac{C_1T\cN_1^\delta(b)}{N}\delta^2\,,
    \end{aligned}
\end{equation*}
for some numerical constant $C>0$ independent of the problem. Using Equation~\eqref{eq:lower_gen} then leads to, for $\eps=\psi(\delta)$:
\begin{equation}\label{eq:Teps}
    T_{\eps}\geq C \frac{dN}{\eps^2\cN_1^\eps(4b)}\,,
\end{equation}
where $C>0$ is a numerical constant independent of the problem. 
\\

\noindent Equation~\eqref{eq:Teps} is obtained for the choice of loss function and distributions defined above, and the functions defined are in $\D(r,b,B)$ for $r=\sqrt{d}$, the fixed weights $b$, and $B^2=1$.
Assumption $\cB_\infty(0,1/2)\subset\cB(0,\sqrt{d}/2)\subset\cX$ is in Theorem~\ref{thm:lower_afo} the assumption that $\NRM{x^0-x_i^\star}\leq r$. In order to obtain a dependency in $r>0$, one simply has to consider the loss $(x,\xi)\mapsto \ell(\frac{\sqrt{d}}{2r}x,\xi)$.
In order to obtain the dependency in $B^2$, we need to modify the distribution, and take $\cD_i'= {\rm Ber}(1/B^2)B^2 \cD_i$. In this case, we have a noise amplitude of order $B^2$ instead of order $1$, and a factor $1/B^2$ appears in the mutual information.

The proof naturally extends to Oracle 1: the mutual information is just $N$ times bigger in that case.

In terms of function distances $d_\cH$ under the assumptions on $\cH$ of Theorem~\ref{thm:lower_distrib}, we have, since the distributions are Bernoulli:
\begin{equation*}
    d_\cH(\cD_i,\cD_j)\leq|\delta_i-\delta_j|\leq b_{ij}.
\end{equation*}

We thus proved the first part of Theorem~\ref{thm:lower_afo}.

\subsection{Proof of the strongly convex and smooth lower bound: second part of Theorem~\ref{thm:lower_afo}}

Let
\begin{equation*}
	\ell(x,\xi)=\frac{1}{2}\NRM{x-\xi}^2\,
\end{equation*}
for $x,\xi\in\R^d$ and, for fixed $\delta>0$ and any $\alpha\in\{-1,1\}^d$:
\begin{equation*}
	g_\alpha(x)=\frac{1}{2d}\sum_{k=1}^d\Big(x_k^2 + 1 - 2\big(\frac{1}{2}+\alpha_k\delta)x_k\Big)\quad,x\in\cX\,.
\end{equation*}
We keep the same notations as last subsection ($\psi(\delta),\rho$). We have:
\begin{equation*}
	\rho(g_\alpha,g_\beta)=\frac{\delta^2}{d}\sum_{k=1}^d |\alpha_k-\beta_k|\,.
\end{equation*}
Similarly to last subsection, this leads to $\psi(\delta)\geq \delta^2/4$  since $\cV$ is a $d/4$-packing of the hypercube.

We keep the same distributions $\cD_1,...,\cD_N$ as last subsection, replacing $\delta_i$ by $\delta_i=(\delta-\sqrt{b_{1i}})^+$.
Agent 1 is still the one that wants to minimize its local generalization error.
The mutual information is thus:
\begin{equation*}
	I(\phi_T,\alpha)\leq \frac{C_1T\cN_1^{\delta}(\sqrt{b})}{N}\delta^2\,.
\end{equation*}
Setting the target precision as $\eps=\delta^2/4$, we obtain:
\begin{equation*}
	T_\eps\geq C' \frac{dN}{\eps\cN_1^\eps(4b)}\,.
\end{equation*}
The loss function and distributions built verify our regularity assumptions for $\mu=1/d$, $L=1/d$, noise $\sigma^2\leq 1$.

We verify that for all $1\leq j,k \leq N$, we have $f_j(x_k^\star)-f_j(x_j^\star)\leq b_{kj}$. We first notice that $x_j^\star=\frac{1}{d}\big(\frac{1}{2}+\delta_j \alpha_l\big)_{1\leq l\leq d}$, so that:
\begin{align*}
	f_j(x_k^\star)-f_j(x_j^\star)&=\frac{1}{d}\NRM{x_j^\star-x_k^\star}^2\\
	&=(\delta_i-\delta_k)^2\\
	&\leq (\sqrt{b_{1j}}-\sqrt{b_{1k}})^2\\
	&\leq |b_{1j}-b_{1k}|\\
	&\leq b_{jk}\,,
\end{align*}
since the weights $b$ verify the triangle inequality. Under the assumptions of Theorem~\ref{thm:lower_distrib} on $\cH$, we have, in terms of distribution-based distances:
\begin{equation*}
    d_\cH(\cD_i,\cD_j)\leq |\delta_i-\delta_j|\leq \sqrt{b_{ij}}\,.
\end{equation*}

The minimum of each $g_\alpha$ is attained at $x^\alpha=\frac{1}{2}+\delta\alpha$, we thus need to assume as in last subsection that $r$ is of order $\sqrt{d}$, and a rescaling leads to the dependency in $r$. The dependency in $\sigma^2$ is obtained as in last subsection.

\section{\emph{All-for-all} lower bound}

Theorem~\ref{thm:lower_afo} immediately yields a Corollary providing lower bounds for the all-to-all problem:



\begin{cor}\label{cor:lower_afa} Let $\eps>0$, $b$ verifying the triangle inequality, $B>0,r>0$. Assume that there exists a method $\cM$ and some $T>0$ such that, for all $((\cD_j)_j,\ell)\in\D(r,b,B)$, $\cM$ returns $(x_1,\ldots,x_N)$ verifying:
\begin{equation*}
    \frac{1}{N}\sum_{i=1}^Nf_i(x_i)-f_i(x_i^\star)\leq \eps\,.
\end{equation*}
where $T$ is the number of stochastic gradients sampled from all agents. Then we have:
\begin{equation*}
    T\geq \frac{Cr^2B^2}{N^2\eps^{2}}\sum_{i=1}^N \frac{1}{\cN_i^\eps(\frac{b}{4N})}\,.
\end{equation*}
Replacing $\D(r,b,B)$ by $\D_{\mu=1/r^2}^{L=1/r^2}(r,b,\sigma)$, we have:
\begin{equation*}
    T\geq \frac{C'r^2\sigma^2}{N\eps}\sum_{i=1}^N \frac{1}{\cN_i^\eps(\frac{b}{4N})}\,.
\end{equation*}
\end{cor}

\section{\emph{All-for-one} upper-bounds\label{app:afo}}

We provide the following relaxation of bias assumption~\ref{hyp:bias2}, a generalization of classical function dissimilarities assumptions \citet{karimireddy_scaffold_2020,chayti2021linear} to our setting.
\begin{enumerate}[label= B.3]
    \item \label{hyp:bias3} For all $1\leq i\leq N$, for any stochastic vector $\lambda\in\R^N$ (non-negative entries that sum to 1), and for all $x\in\R^d$, where $g_\lambda(x)=\sum_{j=1}^N\lambda_j\nabla f_j(x)$: 
    \begin{equation}
        \NRM{\nabla f_i(x)-g_\lambda(x)}^2
        \leq m\NRM{g_\lambda(x)}^2 + \sum_{j=1}^N\lambda_j\tilde{b}_{ij}\,.
        \label{eq:hyp:bias3}
    \end{equation}
\end{enumerate}
In a similar setting ($N$ agents with heterogeneous functions), \citet{karimireddy_scaffold_2020,deng_adaptive_2020,chayti2021linear} assume that for a fixed agent (or for any agent) $i$, $\NRM{\nabla f_i(x)-\nabla \Bar{f}(x)}^2\leq m\NRM{\nabla \Bar{f}(x)}^2+\zeta^2$ (or similar assumptions) where $\Bar{f}$ is the mean of all functions. \eqref{hyp:bias2} is the simplest relaxation of this to the more general setting where all agents want to minimize their local function, while~\eqref{hyp:bias3} is a relaxation that keeps the less-restrictive first term for some $m\geq0$.

We prove the following intermediate result.
\begin{theorem}[WGA] \label{thm:wga} Let $(x^k)_{k\geq0}$ be generated with \eqref{eq:wga}, for some fixed stochastic vector $\lambda\in\R^N$, stepsize $\eta>0$. Let $1\leq i\leq N$ fixed. Let $b\in{\R^+}^{N\times N}$.  Denote $F_i^k=f_i(x^k)-f_i(x_i^\star)$.
\begin{enumerate}[label=\ref{thm:wga}.\alph*]
    \item \label{thm:wga1} Use the asynchronous oracle. Assume that for all $k\geq0$ and $i$, $\NRM{x^k-x_i^\star}\leq D$ for some $D>0$. If $f_i$ is convex, noise assumption~\eqref{hyp:noise2} holds for some $B>0$ and bias assumption~\eqref{hyp:bias2} holds for $\tilde{b}=(b/D)^2$, for any $K>0$ and for $\eta=\sqrt{\frac{D^2}{2KNB^2\sum_{j=1}^N\lambda_j^2}}$:
    \begin{align*}
        \esp{\frac{1}{K}\sum_{0\leq k<K} F_i^k }&\leq \sqrt{\frac{2D^2B^2}{\frac{K}{N}}\sum_{j=1}^N \lambda_j^2} + \sum_{1\leq j\leq N} \lambda_j b_{ij}\,.
    \end{align*}
    \item \label{thm:wga2} Use the synchronous oracle. Assume that all $f_j$ are $\mu$-strongly convex and $L$-smooth (denote $\kappa=L/\mu$), satisfy noise assumption~\ref{hyp:noise1} for some $\sigma>0$, that one of bias assumptions~\eqref{hyp:bias2} or~\eqref{hyp:bias3} holds for $\tilde{b}=\mu b$ (and any $m\geq0$). For any $K>0$ and $\eta=\min(1/(2L), \frac{1}{\mu K}\ln(\frac{F_\lambda^0\mu^2K}{L\sigma^2}\sum_j\lambda_j^2))$ where $F_\lambda^0=\sum_j \lambda_j( f_j(x^0)-f_j(x_j^\star))$, we have:
    \begin{align*}
     \esp{F_i^K}\leq&2(m+1)\kappa\left( F_\lambda^0 e^{-\frac{K}{2\kappa}} + \Tilde{\mathcal{O}}\left(\frac{L\sigma^2}{\mu^2 K} \sum_{1\leq j\leq N} \lambda_j^2\right)\right)+\sum_{1\leq j\leq N} \lambda_j b_{ij}\,.
    \end{align*}
\end{enumerate}
\end{theorem}
This leads to a more general sample complexity in Setting~\ref{setting2}, where we can use \eqref{hyp:bias3}.
\begin{theorem}\label{thm:wga_sample_appendix} Let $\eps>0$, and set $\lambda_j=\frac{\mathds{1}_{\set{b_{ij}<\eps/2}}}{\cN_i^{\eps}(2b)}$, $\cN_i^\eps$ being defined in Theorem~\ref{thm:lower_afo}. Generate $(x^k)_k$ using~\eqref{eq:wga}.
\vspace{-12pt}
\begin{enumerate}[label=\ref{thm:wga_sample_appendix}.\arabic*]
    \item \label{thm:wga_sample1_appendix} \textbf{Under Setting~\ref{setting1}.} Assume that for all $k\geq0$ and $i$, $\NRM{x^k-x_i^\star}\leq D$ for some $D>0$ and bias assumption~\eqref{hyp:bias2} holds for $\tilde{b}=(b/D)^2$. For $K=T_\eps$ and for $\eta=\sqrt{\frac{D^2}{2KNB^2\sum_{j=1}^N\lambda_j^2}}$, we have $\esp{\frac{1}{K}\sum_{0\leq k<K}f_i(x^k)-f_i(x_i^\star)}\leq \eps$ using a total number $T_\eps$ of stochastic gradients from all agents of:
    \begin{equation*}
        T_\eps \leq\frac{4D^2B^2}{\eps^2}\frac{N}{\cN_i^\eps(2b)}\,.
    \end{equation*}
    \vspace{-8pt}
    \item \label{thm:wga_sample2_appendix} \textbf{Under Setting~\ref{setting2}.} Denote $\kappa=L/\mu$, assume that one of bias assumptions~\eqref{hyp:bias2} or~\eqref{hyp:bias3} holds for $\tilde{b}=\mu b$ (and any $m\geq0$). For $K=T_\eps/N$ and $\eta=\min(1/(2L), \frac{1}{\mu K}\ln(\frac{F_\lambda^0\mu^2K}{L\sigma^2}\sum_j\lambda_j^2))$ where $F_\lambda^0=\sum_j \lambda_j( f_j(x^0)-f_j(x_j^\star))$, we have $\esp{f_i(x^K)-f_i(x_i^\star)}\leq\eps$ using a total number $T_\eps$ of stochastic gradients from all agents of:
    \begin{equation*}
        T_\eps \leq \Tilde{\mathcal{O}}\left(\frac{(m+1)\kappa^2\sigma^2}{\mu\eps}\frac{N}{\cN_i^\eps(2b)}\right)\,,
    \end{equation*}
    and of 
    \begin{equation*}
        T_\eps \leq \Tilde{\mathcal{O}}\left(\frac{\kappa\sigma^2}{\mu\eps}\frac{N}{\cN_i^\eps(2b)}\right)\,,
    \end{equation*}
    if $m\leq 1/2$.
\end{enumerate}
\end{theorem}

In Theorem~\ref{thm:wga2} upper-bound, the first term is the optimization term, and vanishes exponentially quickly in front of the second and third ones (if they are non null), and is not reminiscent of our personalized problem. The second is the noise term (referred to as the statistical term), while the third is a bias term, making appear as advertised in the introduction a \emph{bias-variance} trade-off. In Theorem~\ref{thm:wga1}, only the statistical and the bias terms are present. In both cases, a right choice of $\lambda_{ij}$ matches the lower bound of Theorem~\ref{thm:lower_afo} in terms of sample complexity, for $\eps>0$ small enough such that the optimization term vanishes in front of the statistical one. In Theorem~\ref{thm:wga1}, $D$ plays the role of $r$ in the lower bound (and can be more restrictive). Parameters $\mu,L$ in Theorem~\ref{thm:wga2} can take any value, while they are fixed in our lower bounds (with $\kappa=1$).

For the two proofs below, we write
\begin{equation*}
    g_\lambda^k(x)=\sum_{j=1}^N\lambda_jG^k(x)_j\,.
\end{equation*}

\subsection{Proof of Theorem~\ref{thm:wga1}}

Assume first that the $f_j$ are differentiable.
Observe first that under noise assumption~\eqref{hyp:noise2} and with the asynchronous oracle:
\begin{equation*}
    \esp{\NRM{g_\lambda^k(x)}^2}\leq \sum_{j=1}^NN\lambda_j^2B^2\,,\quad x\in\R^d\,.
\end{equation*}
We now proceed following classical SGD analysis, but with biased gradients here. Denote $e_i^k=\esp{\NRM{x^k-x_i^\star}^2}$. We have:
\begin{align*}
    e_i^{k+1}&=e_i^k-2\esp{\langle \sum_{j=1}^N\lambda_j\nabla f_j(x^k),x^k-x_i^\star\rangle} +\eta^2\esp{\NRM{g_\lambda^k}(x^k)}\\
    &\leq e_i^k -2\eta \esp{\langle \nabla f(x^k),x^k-x_i^\star\rangle} + \boldsymbol{2\eta\esp{\langle \nabla f(x^k)- \sum_{j=1}^N\lambda_j\nabla f_j(x^k),x^k-x_i^\star\rangle}} + \eta^2\sum_{j=1}^NN\lambda_j^2B^2\,,
\end{align*}
where the bold term is the bias term, that we bound by $D\sum_j\lambda_j\sqrt{\Tilde{b}_{ij}}=\sum_j\lambda_jb_{ij}$ using bias assumption~\eqref{eq:hyp:bias2} and $\tilde{b}_{ij}=(b_{ij}/D)^2$. Then, since $-2\eta \esp{\langle \nabla f(x^k),x^k-x_i^\star\rangle}\leq -2\eta \E F_i^k$, we have, summing the above inequality for $1\leq k\leq K$:
\begin{equation*}
    \sum_{k=0}^{K-1}F_i^k \leq \frac{e_i^0}{2\eta} + K\frac{\eta}{2} NB^2\sum_{j=1}^N\lambda_j^2 + K\sum_{j=1}^N\lambda_jb_{ij}\,.
\end{equation*}
We obtain the desired result by dividing by $K$ and for the choice of $\eta$ as in Theorem~\ref{thm:afa1}.
In the case where we have access to subdifferentials only, the proof stays identical (as in Lemma~\ref{lem:sgd}).

\subsection{Proof of Theorem~\ref{thm:wga2}}

Denoting $f^\lambda(x)=\sum_{j=1}^N\lambda_jf_j(x)$, observe that $f^\lambda$ is $\mu$-strongly convex and $L$-smooth, and for all $x\in\R^d$, under the synchronous oracle and noise assumption~\eqref{hyp:noise1}:
\begin{equation}
    \begin{aligned}
    &\esp{g_\lambda^k(x)}=\nabla f^\lambda(x)\,,\\
    &\esp{\NRM{g_\lambda^k(x)-\nabla f^\lambda(x)}^2}\leq \sigma^2\sum_{j=1}^N\lambda_j^2\,.
    \end{aligned}
\end{equation}
Applying Lemma~\ref{lem:sgd_sc} (SGD under strong convexity and smoothness assumptions), we have, where $x^\lambda$ minimizes $f^\lambda$:
\begin{equation*}
    f^\lambda(x^k)-f^\lambda(x^\lambda)\leq F_\lambda^0 e^{-\frac{ k}{2\kappa}} + \Tilde{\mathcal{O}}\left(\frac{L\sigma^2}{\mu^2 k} \sum_{1\leq j\leq N} \lambda_j^2\right)\,.
\end{equation*}
Then, under either bias assumption~\eqref{hyp:bias2} or~\eqref{hyp:bias3}, we have:
\begin{align*}
    \NRM{\nabla f_i(x^k)}^2&\leq 2\NRM{\nabla f_i(x^k)-\nabla f^\lambda(x^k)}^2+2\NRM{\nabla f^\lambda(x^k)}^2\\
    &\leq 2 ( m\NRM{\nabla f^\lambda(x^k)}^2 + \sum_{j=1}^N \lambda_{ij}\tilde{b}_{ij})+ 2\NRM{\nabla f^\lambda(x^k)}^2\\
    &\leq 2\sum_{j=1}^N\lambda_{ij}\tilde{b}_{ij}+ 4L(m+1) (f^\lambda(x^k)-f^\lambda(x^\lambda))\,.
\end{align*}
We conclude using $f_i(x^k)-f_i(x_i^\star)\leq \frac{1}{2\mu}\NRM{\nabla f_i(x^k)}^2$ and $\tilde{b}_{ij}=\mu b_{ij}$.

This result is valid under any $m\geq0$, and gives the first part of Theorem~\ref{thm:wga_sample2_appendix}. If $m\leq 1/2$, one can use Theorem 1 in~\citet{chayti2021linear} with the set of agents $j$ such that $b_{ij}\leq \eps$, agent $0$ is agent $i$, with $\alpha=1/\cN_i^\eps(2b)$, to obtain the sample complexity without the $\kappa$ factor.

\section{\emph{All-for-all} upper-bounds}

We first begin with the following simple lemma.
\begin{lemma}\label{lem:bias}
    If $\Lambda$ is a stochastic matrix and if bias assumption~\eqref{hyp:bias1} holds ($f_i(x_j^\star)-f_i(x_i^\star)\leq b_{ij}$ for all $1\leq i,j\leq N$):
    \begin{equation*}
            f^\Lambda(x^\Lambda)-\Bar{f}(x^\star)\leq \frac{1}{N}\sum_{1\leq i,j\leq N}\lambda_{ij}b_{ij}\,.
    \end{equation*}
\end{lemma}
\begin{proof}
Writing the optimality of $x^\Lambda$ gives:
\begin{align*}
    f^\Lambda(x^\Lambda)&\leq f^\Lambda(x^\star)\\
    &=\frac{1}{N}\sum_if_i(\sum_j\lambda_{ij}x_j^\star)\\
    &\leq \frac{1}{N}\sum_{1\leq i,j\leq N}\lambda_{ij}f_i(x_j^\star)\,,
\end{align*}
where we used convexity of each $f_i$. Then, subtracting $\Bar{f}(x^\star)$ and using stochasticity of $\Lambda$:
\begin{align*}
    f^\Lambda(x^\Lambda)-\Bar{f}(x^\star)\leq \frac{1}{N}\sum_{1\leq i,j\leq N}\lambda_{ij}(f_i(x_j^\star)-f_i(x_i^\star))\,.
\end{align*}
\end{proof}

\subsection{Proof of Theorem~\ref{thm:afa1}}

Using Lemma~\ref{lem:bias} and the proof sketch ($f^\Lambda(y^k)=\Bar{f}(x^k)$ and the bias variance decomposition):
\begin{equation*}
    \Bar{f}(x^k)-\Bar{f}(x^\star)\leq f^\Lambda(y^k)-f^\Lambda(x^\Lambda) + \frac{1}{N}\sum_{1\leq i,j\leq N}\lambda_{ij}b_{ij}\,,
\end{equation*}
where $(y^k)$ is generated using simple SGD on $f^\Lambda$:
\begin{equation*}
    y^{k+1}=y^k-\eta\nabla G^k_\Lambda(y^k)\,,
\end{equation*}
where
\begin{equation*}
    G^k_\Lambda(y)=\big(\lambda_{i_kj}g_{i_k}^k((\Lambda y^k)_{i_k}))\big)_{1\leq j\leq N}\,.
\end{equation*}
$f^\Lambda$ is convex, and we have under noise assumption~\eqref{hyp:noise2} for the asynchronous oracle, under differentiable assumptions:
\begin{align*}
    &\esp{G^k_\Lambda(y)}=\nabla f^\Lambda(y)\,,\\
    &\esp{\NRM{G^k_\Lambda(y)}^2}\leq \frac{B^2}{N} \sum_{1\leq i,j\leq N}\lambda_{ij}^2\,.
\end{align*}
If we only have subdifferentials, we still have $\esp{G^k_\Lambda(y)}\in\partial f^\Lambda(y)$. Using Lemma~\ref{lem:sgd} (SGD under these assumptions), we obtain the desired result.

\subsection{Proof of Theorem~\ref{thm:afa2}}

Under noise assumption~\eqref{hyp:noise1} and for the synchronous oracle, we stil have:
\begin{equation*}
    y^{k+1}=y^k-\eta\nabla G^k_\Lambda(y^k)\,,
\end{equation*}
for
\begin{equation*}
    G^k_\Lambda(y)=\frac{1}{N}\big(\sum_{i=1}^N\lambda_{ij}g_{i}^k((\Lambda y^k)_{i}))\big)_{1\leq j\leq N}\,,
\end{equation*}
that verifies:
\begin{align*}
    &\esp{G^k_\Lambda(y)}=\nabla f^\Lambda(y)\,,\\
    &\esp{\NRM{G^k_\Lambda(y)-\nabla f^\lambda(y)}^2}\leq \frac{\sigma^2}{N^2} \sum_{1\leq i,j\leq N}\lambda_{ij}^2\,.
\end{align*}
The function $f^\Lambda$ is however not necessarily strongly convex. However, since $\nabla^2f^\Lambda(y)=\Lambda^\top \nabla^2\Bar{f}(\Lambda y)\Lambda$ and $\Bar{f}$ is $L/N$-smooth and $\mu/N$-strongly convex, $f^\Lambda$ is $L/N$-relatively smooth and $\mu/N$-relatively strongly convex \citep{bauschke2017relative} with respect to $\frac{1}{2}\NRM{y}_W^2=\frac{1}{2}y^\top Wy$. Note also that the spectral radius of $W$ is $1$, since $\Lambda$ is stochastic. Instead of using stochastic Bregman gradient descent (\emph{e.g.}~\citet{pmlr-v139-dragomir21a}), we use Lemma~\ref{lem:sgd_sc}: classical SGD that naturally generalizes to relative smoothness and strong convexity assumptions, when the mirror map is quadratic.

\subsection{Time-adaptive variant\label{app:time_adaptive}}

\begin{algorithm}[h]
\caption{\emph{All-for-all} algorithm: Time-adaptive variant}
\label{algo:afa_time}
\begin{algorithmic}[1]
\STATE Stepsizes $(\eta_k)_{k\geq0}$, matrixes $(W^{(k)})_{k\geq0}$ 
\vspace{0.5ex}
\STATE Initialization $x_1^0=\ldots=x_N^0\in\R^d$
\vspace{0.5ex}
\FOR{$k=0,1,2,\ldots$}
\vspace{0.5ex}
\STATE Agents $j\in S^k$ (activated agents) compute stochastic gradients $g_j^k(x_j^k)$ and broadcast it to all agents $i$ such that $W^{(k)}_{ij}>0$.
\vspace{0.5ex}
\STATE For $i=1,\ldots,N$, update $$x_i^{k+1}=x_i^k-\eta_k\sum_{j\in S^k}W^{(k)}_{ij}g^k_j(x^k_j)$$
\ENDFOR  
\end{algorithmic}
\end{algorithm}

\begin{proposition}[Time-adaptive \emph{all-for-all}]
    Assume that the same assumptions as in Theorem~\ref{thm:afa1} hold.
    For any $\eps>0$, denote $\Lambda_\eps=(\frac{\mathds{1}_{2b_{ij}\leq\eps}}{\cN_i^\eps(2b)}$, $K_\eps=\frac{4D^2B^2}{\eps^2}\sum_{i=1}^N\frac{1}{\cN_i^\eps(2b)}$ and $\eta(\eps)=\frac{2ND^2}{K_\eps B^2\sum_i1/\cN_i^\eps(2b)}$. For any $k\geq0$, there exists $p_k\geq 0$ such that $\sum_{q=0}^{p_k-1}T_{2^{-q}}\leq k <\sum_{q=0}^{p_k}T_{2^{-q}}$, and choose:
    \begin{align*}
        \eta_k=\eta(2^{-p_k})\,,\quad W^{(k)}=\Lambda_{2^{-p_k}}\Lambda_{2^{-p_k}}^\top\,.
    \end{align*}
    Then, for any $\eps>0$, the iterates generated by Algorithm~\ref{algo:afa_time} reach averaged precision $\eps$ for a number of data item sampled from personal distribution of:
    \begin{equation*}
                T_\eps\leq \frac{16D^2B^2}{\eps^2}\sum_{i=1}^N\frac{1}{\cN_i^\eps(2b)}\,,
    \end{equation*}
\end{proposition}
Similarly, under strong-convexity and smoothness assumptions as in Theorem~\ref{thm:afa2}, one would obtain the corresponding sample complexity, up to a factor 2.

\subsection{Setting~\ref{setting2} under $\mu=0$\label{app:afa_convex}}

\begin{theorem}[All-for-all, convex case]\label{thm:afa_convex}
	Let $K>0$, $\eta>0$, and $W$ a symmetric non-negative random matrix of the form $W=\Lambda\Lambda^\top$ for some stochastic matrix $\Lambda=(\lambda_{ij})_{1\leq i,j\leq N}$.
	Let $(x_i^k)_{k\geq0,1\leq i\leq N}$ be generated with Algorithm~\ref{algo:afa}. Assume that bias assumption~\eqref{hyp:bias1} holds for some $(b_{ij})$.
	Under Setting \ref{setting2}, for $\mu=0$ (convex case), if $\eta=\frac{N^{3/2}D}{\sigma\sqrt{\sum_{i,j}\lambda_{ij}^2}}\wedge \frac{N}{2L}$, where $\NRM{x_i^0-x_i^\Lambda}\leq D$ for all $i$, we have:		
		\vspace{-5pt}
		\begin{align*}
			\esp{F^K}&\leq \frac{2LD^2}{K}+\sqrt{\frac{2D^2\sigma^2}{KN}\sum_{1\leq i,j\leq N}\lambda_{ij}^2} +\sum_{1\leq i,j\leq N}\lambda_{ij}b_{ij}\,.
			\vspace{-5pt}
		\end{align*}
	Consequently, for $\eps>0$, and a choice of $W=\Lambda\Lambda^\top$ for $\lambda_{ij}=\frac{\mathds{1}_\set{b_{ij}<\eps/2}}{\cN_i^\eps(2b)}$ and under the same assumptions, the \emph{all-for-all} algorithm reaches $\esp{F^K}\leq \eps$ for a total number $T_\eps$ of data item sampled from personal distributions of:
	\begin{equation*}
		T_\eps\leq \max\left( \frac{8D^2\sigma^2}{\eps^2}\sum_{i=1}^N \frac{1}{\cN_i^\eps(2b)},\, \frac{4NLD^2}{\eps}\right)
	\end{equation*}
	\end{theorem}
\begin{proof}
Combining Lemma~\ref{lem:sgd_cs} with the noise variance and regularity parameters of $F^\Lambda$.
\end{proof}

\section{Stochastic Optimization Toolbox}

\subsection{SGD under strongly convex and smooth assumptions}

\begin{lemma}[SGD, s.c. and smooth]\label{lem:sgd_sc} Define $\NRM{x}_A^2=x^\top A x$ for some non-negative and symmetric matrix $A$. Let $f:\cX\to\R$ $\mu$-relatively strongly convex and $L$-relatively smooth with respect to $\frac{1}{2}\NRM{x}_A^2$. Let $(f_t,g_t)_{t\geq0}$ be first order oracle calls such that for all $t\geq0$:
	\begin{equation*}\forall x\in\cX\,,\quad
		\left\{\begin{aligned}
			& \esp{f_t(x)}=f(x)\,,\\
			&\esp{g_t(x)}=\nabla f(x)\,,\\
			&\esp{\NRM{g_t(x)-\nabla f(x)}^2}\leq \sigma^2\,,
		\end{aligned}\right.
	\end{equation*}
	for some $\sigma>0$.
	Let $L_A$ be the largest eigenvalue of $A$, and assume that $L_A\leq1$ (our result generalizes to any $L_A$).
	Let $(x_t)_{t\geq0}$ be generated with:
	\begin{equation*}
		\forall t\geq0\,,\quad x^{t+1}=x^t-\eta g_t(x^t)\,,
	\end{equation*}
	for a fixed stepsize $\frac{1}{2L}\geq\eta>0$, and assume that all the iterates lie in $\cX$. Assume that $f$ is minimized over $\cX$ at some interior point $x^\star$. We have for any $T>0$:
	\begin{equation*}
		\esp{f(x^T)-f(x^\star)}\leq e^{-\eta\mu T}\big(f(x^0)-f(x^\star)\big) + \frac{\eta L\sigma^2}{\mu}\,.
	\end{equation*}
	For fixed $T>0$, setting $\eta=\min\big(1/(2L),\frac{1}{\mu T}\ln(\frac{f_0\mu^2T}{L\sigma^2})\big)$ gives:
	\begin{equation*}
		\esp{f(x^T)-f(x^\star)}\leq e^{-\frac{\mu}{2L}T}\big(f(x^0)-f(x^\star)\big)+ \frac{L\sigma^2}{\mu^2T}\ln\big(\frac{f_0\mu^2T}{L\sigma^2}\big)\,.
	\end{equation*}
	Thus, for fixed target precision $\eps>0$, using stepsize $\eta_\eps=\min\left(\frac{\mu\eps}{2L\sigma^2},\frac{1}{2L}\right)$ and setting $T_\eps=\lceil \ln\big(\eps^{-1}(f(x^0)-f(x^\star))\big)\frac{1}{\eta_\eps\mu}\rceil$, we have:
	\begin{equation*}
		f\left(\frac{1}{T_\eps}\sum_{t<T_\eps}x^t\right)-f(x^\star)\leq \eps\,,
	\end{equation*}
	with a number of oracle calls
	\begin{equation*}
		T_\eps\leq \max\left( \frac{2L\sigma^2}{\eps\mu^2},\frac{2L}{\mu}\right)\ln\big(\eps^{-1}(f(x^0)-f(x^\star))\big)\,.
	\end{equation*}
\end{lemma}
\begin{proof}
	For some $t\geq0$, denoting $f_t=\esp{f(x^{t+1})-f(x^\star)}$, using relative smoothness, unbiasedness of the stochastic gradients and then relative strong convexity:
	\begin{align*}
		f_{t+1}-f_t&\leq -\eta \esp{\NRM{\nabla f(x^t)}^2} + \frac{\eta^2 L}{2}\esp{\NRM{g_t}_A^2}\\
		&\leq -\eta \esp{\NRM{\nabla f(x^t)}^2} + \frac{\eta^2 LL_A}{2}\esp{\NRM{g_t}^2}\\
		&\leq -\eta\left(1-\frac{\eta LL_A}{2}\right) \esp{\NRM{\nabla f(x^t)}^2} +\frac{\eta^2LL_A\sigma^2}{2}\,.
	\end{align*}
Using relative strong convexity of $f$, we have: 
\begin{align*}
    \NRM{\nabla f(x^t)}^2&\geq \frac{1}{L_A}\NRM{\nabla f(x^t)}_A^2\\
    &\leq \frac{2\mu}{L_A}f_t\,,
\end{align*}
yielding, for $\eta<1/(LL_A)$
\begin{equation*}
    f_{t+1}-f_t\leq -2\eta\frac{\mu}{L_A}f_t+\frac{\eta^2LL_A\sigma^2}{2}\,.
\end{equation*}
Then, for some $T>0$ and since $L_A\leq 1$, sum the above inequality multiplied by $\left(1-\eta\mu\right)^{-t-1}$:
\begin{align*}
	\sum_{0\leq t\leq T-1} \left(1-\eta\mu\right)^{-t-1}f_{t+1}-\left(1-\eta\mu\right)^{-t}f_t &\leq \frac{\eta^2L\sigma^2}{2} \sum_{0\leq t\leq T-1} \left(1-\eta\mu\right)^{-t-1}\\
	&\leq \frac{\eta^2L\sigma^2}{2} \frac{\left(1-\eta\mu\right)^{-t-1}}{\eta\mu}\,,
\end{align*}
leading to the desired result.
\end{proof}

\subsection{SGD under smoothness and convexity assumptions}

\begin{lemma}[SGD, convex and smooth]\label{lem:sgd_cs} Let $f:\cX\to\R$  convex and $L$-smooth. Let $(f_t,g_t)_{t\geq0}$ be first order oracle calls such that for all $t\geq0$:
	\begin{equation*}\forall x\in\cX\,,\quad
		\left\{\begin{aligned}
			& \esp{f_t(x)}=f(x)\,,\\
			&\esp{g_t(x)}=\nabla f(x)\,,\\
			&\esp{\NRM{g_t(x)-\nabla f(x)}^2}\leq \sigma^2\,,
		\end{aligned}\right.
	\end{equation*}
	for some $\sigma>0$. Let $(x_t)_{t\geq0}$ be generated with:
	\begin{equation*}
		\forall t\geq0\,,\quad x^{t+1}=x^t-\eta g_t(x^t)\,,
	\end{equation*}
	for a fixed stepsize $\frac{1}{2L}\geq\eta>0$, and assume that all the iterates lie in $\cX$. Assume that $f$ is minimized over $\cX$ at some interior point $x^\star$. 
	Denote, for $T>0$, $\bar{x}^T=\frac{1}{T}\sum_{0\leq t<T}x^t$.
	We have for any $T>0$:
	\begin{equation*}
		\esp{f(\bar{x}^T)-f(x^\star)}\leq \frac{\NRM{x^0-x^\star}^2}{\eta T} + \eta\sigma^2\,.
	\end{equation*}
	For fixed $T>0$, setting $\eta=\min\Big(1/(2L)),\frac{\NRM{x^0-x^\star}}{\sigma\sqrt{T}}\Big)$ gives:
	\begin{equation*}
		\esp{f(x^T)-f(x^\star)}\leq \frac{2\NRM{x^0-x^\star}\sigma}{\sqrt{T}} + \frac{2L\NRM{x^0-x^\star}^2}{T}\,.
	\end{equation*}
	Thus, for fixed target precision $\eps>0$,  we have:
	\begin{equation*}
		f\left(\frac{1}{T_\eps}\sum_{t<T_\eps}x^t\right)-f(x^\star)\leq \eps\,,
	\end{equation*}
	with a number of oracle calls
	\begin{equation*}
		T_\eps\leq 4\NRM{x^0-x^\star}^2\big(4\sigma^2\eps^{-2}+L\eps^{-1}\big) \,.
	\end{equation*}
\end{lemma}
\begin{proof}
	Denote $e_t=\esp{\NRM{x^t-x^\star}^2}$. We have:
	\begin{align*}
		e_{t+1}-e_t&\leq -2\eta\langle \nabla f(x^t),x^t-x^\star\rangle + \eta^2\esp{\NRM{g_t^2}}\\
		&\leq -2\eta f_t +\eta^2(\esp{\NRM{\nabla f(x^t)}^2}+\sigma^2)\\
		&\leq -2\eta(1-L\eta)f_t +\eta^2\sigma^2\\
		&\leq -\eta f_t +\eta^2\sigma^2\,,
	\end{align*}
	where we used convexity and smoothness assumptions. Then, summing:
	\begin{equation*}
		\sum_{t<T}f_t \leq \frac{e_0-e_T}{\eta} + T\eta\sigma^2\,,
	\end{equation*}
	leading to the desired result.
\end{proof}

\subsection{SGD under convex assumptions}

\begin{lemma}[SGD, convex]\label{lem:sgd} Let $f:\cX\to\R$ convex. Let $(f_t,g_t)_{t\geq0}$ be first order oracle calls such that for all $t\geq0$:
	\begin{equation*}\forall x\in\cX\,,\quad
		\left\{\begin{aligned}
			& \esp{f_t(x)}=f(x)\,,\\
			&\esp{g_t(x)}\in\partial f(x)\,,\\
			&\esp{\NRM{g_t(x)}^2}\leq B^2\,,
		\end{aligned}\right.
	\end{equation*}
	for some $B>0$. For $p_\cX$ the projection on the convex set $\cX$, let $(x_t)_{t\geq0}$ be generated with:
	\begin{equation*}
		\forall t\geq0\,,\quad x^{t+1}=p_\cX\left(x^t-\eta g_t(x^t)\right)\,,
	\end{equation*}
	for a fixed stepsize $\eta>0$. Assume that $f$ is minimized over $\cX$ at some interior point $x^\star$. We have for any $T>0$ and for $\eta=\sqrt{\frac{2D^2}{B^2T}}$ for $D^2\geq \NRM{x^0-x^\star}^2$:
	\begin{equation*}
		f\left(\frac{1}{T}\sum_{t<T}x^t\right)-f(x^\star)\leq \sqrt{\frac{2}{T}}BD\,.
	\end{equation*}
	Thus, for fixed target precision $\eps>0$, setting $T_\eps=\lceil 2\eps^{-2}B^2D^2\rceil$ and using stepsize $\eta_\eps=\frac{\sqrt{2}\eps}{B^2}$, we have:
	\begin{equation*}
		f\left(\frac{1}{T_\eps}\sum_{t<T_\eps}x^t\right)-f(x^\star)\leq \eps\,.
	\end{equation*}
\end{lemma}
\begin{proof}
For any $y\in\cX$, using properties of $p_\cX$:
\begin{align*}
	\NRM{x^{t+1}-y}^2&=\NRM{p_\cX\left(x^t-\eta g_t(x^t)\right)-y}^2\\
	&\leq \NRM{x^t-\eta g_t(x^t)-y}^2\\
	&=\NRM{x^t-y}^2 + \eta^2\NRM{g_t(x^t)^2} -2 \eta\langle g_t(x^t),x^t-y\rangle\,.
\end{align*}
Taking the mean and $y=x^\star$, with $e_t=\esp{\NRM{x^t-x^\star}^2}$ and $f_t=\esp{f(x^t)-f(x^\star)}$:
\begin{align*}
	e_{t+1}&\leq e_t +\eta^2B^2 -2\eta \langle \nabla f(x^t),x^t-y\rangle\\
	&\leq e_t +\eta^2B^2 -2\eta f_t \,,
\end{align*}
where we used convexity of $f$. Thus, $f_t\leq \frac{e_t-e_{t+1}}{2\eta} + \frac{\eta B^2}{2}$, and by summing this inequality for $0\leq t<T$:
\begin{equation*}
	\sum_{t<T}f_t \leq \frac{e_0-e_T}{2\eta} + \frac{\eta T B^2}{2} \leq \frac{e_0}{\eta} + \frac{\eta T B^2}{2}\,.
\end{equation*}
Thus, for $\eta=\sqrt{\frac{2e_0}{B^2T}}$, we have the result using convexity of $f$.
\end{proof}
\end{document}